\documentclass{amsart}
\usepackage{amsbsy,amssymb,amscd,amsfonts,latexsym,amstext,delarray, amsmath,graphicx,color,caption,enumerate,lscape,stackrel, stackengine,mathtools,leftidx}
\usepackage{graphicx}
\usepackage{caption}
\usepackage{subcaption}
\usepackage{epsfig}
\usepackage{tikz,tikz-cd}
\usetikzlibrary{matrix,calc,arrows,decorations.markings,patterns,calc,decorations.pathmorphing}
\input xy

\xyoption{all}
\usepackage{euscript}
\usepackage{multirow}
\usepackage{etex, pictexwd,dcpic}
\usepackage[makeroom]{cancel}

\newtheorem{theorem}{Theorem}
\newtheorem{definition}[theorem]{Definition}
\newtheorem{proposition}[theorem]{Proposition}
\newtheorem{lemma}[theorem]{Lemma}
\newtheorem{corollary}[theorem]{Corollary}
\newtheorem{example}[theorem]{Example}

\newtheorem{remark}[theorem]{Remark}
\def\F{{\mathcal F}}

\newcommand{\db}[1]{\mathopen{\{\!\!\{}#1\mathopen{\}\!\!\}}}

\newcommand\oast{\stackMath\mathbin{\stackinset{c}{0ex}{c}{0ex}{\ast}{\bigcirc}}}

\newcommand{\git}{\mathbin{
  \mathchoice{/\mkern-6mu/}
    {/\mkern-6mu/}
    {/\mkern-5mu/}
    {/\mkern-5mu/}}}

\tikzset{
->-/.style={decoration={
  markings,
  mark=at position .5 with {\arrow{>}}},postaction={decorate}},
->--/.style={decoration={
  markings,
  mark=at position .3 with {\arrow{>}}},postaction={decorate}},
-->-/.style={decoration={
  markings,
  mark=at position .7 with {\arrow{>}}},postaction={decorate}},
-<-/.style={decoration={
  markings,
  mark=at position .5 with {\arrow{<}}},postaction={decorate}},
-<--/.style={decoration={
  markings,
  mark=at position .3 with {\arrow{<}}},postaction={decorate}},
--<-/.style={decoration={
  markings,
  mark=at position .7 with {\arrow{<}}},postaction={decorate}},
--->-/.style={decoration={
  markings,
  mark=at position .85 with {\arrow{>}}},postaction={decorate}},
---->-/.style={decoration={
  markings,
  mark=at position .95 with {\arrow{>}}},postaction={decorate}},
----->-/.style={decoration={
  markings,
  mark=at position .97 with {\arrow{>}}},postaction={decorate}},
subset/.style={
    draw=none,
    every to/.append style={
      edge node={node [sloped, allow upside down, auto=false]{$\subset$}}}
  },
equal/.style={
    draw=none,
    every to/.append style={
      edge node={node [sloped, allow upside down, auto=false]{$=$}}}
  }
assign/.style={
    draw=none,
    every to/.append style={
      edge node={node [sloped, allow upside down, auto=false]{$=$}}}
  }
}
\tikzset{
  symbol/.style={
    draw=none,
    every to/.append style={
      edge node={node [sloped, allow upside down, auto=false]{$#1$}}}
  }
}

\title{Noncommutative Networks on a Cylinder}

\author{S.~Arthamonov}
\address{S.A.: Department of Mathematics, University of California, Berkeley, CA 94720 USA}
\email{artamonov@berkeley.edu}

\author{N.~Ovenhouse}
\address{N.O.: School of Mathematics, University of Minnesota, Minneapolis, MN 55455 USA}
\email{ovenh001@umn.edu}

\author{M.~Shapiro}
\address{M.S.: Department of Mathematics, Michigan State University, East Lansing, MI 48824 USA and National Research University Higher School of Economics, Moscow, Russia}
\email{mshapiro@msu.edu}

\begin{document}

\begin{abstract}
    In this paper a double quasi Poisson bracket in the sense of Van den Bergh is constructed on the space of
    noncommutative weights of arcs of a directed graph embedded in a disk or cylinder $\Sigma$, which gives rise to the quasi Poisson bracket of G.Massuyeau and V.Turaev
    on the group algebra $\mathbf k\pi_1(\Sigma,p)$ of the fundamental group of a surface based at $p\in\partial\Sigma$.
    This bracket also induces a noncommutative Goldman Poisson bracket on the \textit{cyclic space} $\mathcal C_\natural$, which is a $\mathbf k$-linear space of unbased loops.
    We show that the induced double quasi Poisson bracket between boundary measurements can be described via noncommutative $r$-matrix formalism.
    This gives a more conceptual proof of the result of \cite{Ovenhouse'2020} that traces of powers of Lax operator form an infinite collection of
    noncommutative Hamiltonians in involution with respect to noncommutative Goldman bracket on $\mathcal C_\natural$.
\end{abstract}

\maketitle

\section{Introduction}

The current manuscript is obtained as a continuation of papers \cite{Ovenhouse'2020,DiFrancesco'2009,DiFrancesco'2015,BerensteinRetakh'2011}
where the authors develop noncommutative generalizations of discrete completely integrable dynamical systems
and \cite{BerensteinRetakh'2018, GoncharovKontsevich'2021} where a large class of noncommutative cluster algebras was constructed.
Cluster algebras were introduced in \cite{fomin2002cluster} by S.Fomin and A.Zelevisnky in an effort to describe the (dual) canonical basis
of universal enveloping algebra ${\mathcal U}({\mathfrak b})$, where $\mathfrak b$ is a Borel subalgebra of a simple complex Lie algebra $\mathfrak g$.
Cluster algebras are commutative rings of a special type, equipped with a distinguished set of generators (cluster variables) subdivided into overlapping subsets
(clusters) of the same cardinality subject to certain polynomial relations (cluster transformations).
In the last twenty years the theory of cluster algebras experienced an explosive development motivated by discovered connections of cluster theory with topology,
integrable systems, theory of positivity, representation theory, etc.

One of the most exciting development is an application of cluster theory to integrable systems. It was shown, for instance,  in \cite{okubo2013discrete}
that discrete Hirota integrable systems are closely related to special sequences of cluster transformations (cluster evolution).
On the other hand, A.Postnikov introduced in \cite{Postnikov'2006} a convenient way to describe some families of cluster coordinates on Grassmannian
in terms of oriented graphs embedded in a disk with weighted faces. In \cite{GSV'2009} and \cite{GSV'2012} the third author, along with M.Gekhtman and A.Vainshtein,
defined a family of Poisson brackets on the space of weights of Postnikov's directed graphs, and showed that the entries of Postnikov's \emph{boundary measurement matrix}, $B$, satisfy Sklyanin $r$-matrix Poisson relations:
\[ \{B(\lambda), B(\mu)\} = \left[ r(\lambda,\mu), \, B(\lambda) \otimes B(\mu) \right] \]
Here, $r(\lambda,\mu)$ is the standard trigonometric $r$-matrix for $SL(N)$.
One of the main results of the current paper is to give a noncommutative generalization of this statement.

A generalization of Postnikov's construction to oriented graphs embedded in a torus leads to
coordinates on affine group. This case is equivalent to considering dimer models on a bipartite graph $\Gamma$ on a torus considered in
\cite{goncharov2013dimers}, which gives rise to a cluster integrable system. The phase space contains, as an open dense subset, the moduli space of line bundles
with connection on the graph $\Gamma$. The complexification of the phase space is birationally isomorphic to a finite cover of the Beauville complex algebraic
integrable system related to a toric surface. In a particular case of a square bipartite graph on a torus, the construction above is equipped additionally with
a celebrated pentragram discrete integrable system commuting with Hamiltonians of the cluster integrable system.

The pentagram map was introduced earlier by R.Schwartz \cite{Schwartz'1992} in a completely different geometric model  as a map on the space of (``twisted'') polygons in the projective plane.
The pentagram map is  equivariant with respect to projective transformations, hence one can consider it on the projective classes of twisted polygons. Iterations of pentagram map generate discrete dynamical system on the phase space of projective classes of twisted polygons in $\mathbb RP^2$.
It was proved in \cite{OvsienkoSchwartzTabachnikov'2010,GSTV'2016} that the pentagram map is completely integrable in the Liouville sense;
i.e., the space of projective classes of polygons can be equipped with a Poisson structure invariant under the pentagram map, and there are a sufficient number of integrals in involution with respect to this Poisson bracket. Using \cite{glick2011pentagram} the pentagram
map and some of its generalizations were represented as a cluster dynamics of weights for a square bipartite graph on a cylinder in \cite{GSTV'2016} where also the complete integrability of all these systems are proved.

Later,  G.Mar\'{i}-Beffa and R.Felipe \cite{MariBeffaFelipe'2019} generalized the construction of the pentagram map to the space of twisted polygons in Grassmannians and found the Lax form of the pentagram transformation, and A.Izosimov \cite{Izosimov'2018} described an invariant Poisson bracket for this dynamical system and showed that the pentagram map is an example of refactorizational dynamics in an affine group.

The Grassmann pentagram map was reformulated by N.Ovenhouse in terms of a discrete dynamics on the space of matrix-valued weights on the same graphs embedded in a torus  \cite{Ovenhouse'2020}. The corresponding dynamics was generalized to the space of weight with values in a noncommutative free division domain $\F$.

In particular, he constructed an invariant noncommutative (Goldman type) Poisson bracket on the cyclic space $\F/[\F,\F]$ and found an infinite sequence of integrals in involution with respect to the Poisson bracket in the form of traces of powers of the Lax matrix.
One of the main difficulties of the approach in \cite{Ovenhouse'2020} came from the fact that the noncommutative Poisson bracket was defined with values in the cyclic space only, which is not an algebra, and therefore the Leibniz identity is not applicable.
To prove involutivity of the infinite family of hamiltonians, Ovenhouse used clever topological arguments using the topological nature of this Goldman type Poisson bracket.
Notice that in commutative case (for the actual pentagram map) the involutivity of traces of powers of the Lax matrix (i.e., coefficients of the spectral curve) follows for free from the fact that the corresponding invariant Poisson bracket is an $r$-matrix bracket (\cite{GSV'2012,GSTV'2016}).
It is a classical result  \cite{ReymanSemenov-Tian-Shansky'1994} that coefficients of the spectral curve are in involution for an $r$-matrix Poisson bracket.

The main goal of the current paper is to introduce an $r$-matrix formalism for the corresponding noncommutative Poisson bracket.
Based on the previous works \cite{Arthamonov'2018, ArthamonovRoubtsov, MassuyeauTuraev'2014, VandenBergh'2008}, we introduce in this paper an $r$-matrix noncommutative double bracket in the sense of Van den Bergh on the space of noncommutative weights of a directed graph embedded in a disk, or a cylinder. The main advantage of the double algebra is that the corresponding double bracket satisfies Leibniz relations which make computations much more transparent. The $r$-matrix formalism allows to give more algebraic and conceptual proofs of the results of \cite{Ovenhouse'2020}.

The plan of the paper is the following.

In section 2 the double quasi Poisson bracket on the space of open arcs of a graph on a general 2D surface is introduced.
In section 3 the expressions for the double bracket on graphs in a disk or a cylinder are computed.
In section 4 its $r$-matrix formulation is defined.
In section 5 we define an infinite family of Hamiltonians in involution and compute explicit Lax form of the corresponding continuous Hamiltonian system. In section 6 we describe sufficient conditions for an $r$-matrix to define a double quasi Poisson bracket. Finally, in section 7 we show that some of the classical results about refactorization dynamics are generalized  to noncommutative case.

\section*{Acknowledgements}

S.A. is grateful to A.~Alexeev, A.~Berenstein, and M.~Kontsevich for many fruitful discussions and to V.~Retakh and V.~Roubtsov
for thesis advising \cite{Arthamonov'2018} and long collaboration \cite{ArthamonovRoubtsov}, which resulted in starting formulas
of Section \ref{sec:DoubleBracketForOpenArcs}. S.A. is also grateful to IHES, Michigan State University, University of Angers,
and University of Geneva for hospitality during my visits.
M.S. is grateful to the following institutions and programs: Research Institute for Mathematical Sciences, Kyoto (Spring 2019),
Research in Pairs Program at the Mathematisches Forschungsinstitut Oberwolfach (Summer 2019), Mathematical Science Research Institute,
Berkeley (Fall 2019) for their hospitality and outstanding working conditions they provided. S.A. was partially supported by RFBR-18-01-00926 and RFBR-19-51-50008-YaF.
N.O. was partially supported by the NSF grant DMS-1745638.
M.S. was partially supported by the NSF grant DMS-1702115 and partially supported by International Laboratory of Cluster Geometry NRU HSE, RF Government grant, ag. № 075-15-2021-608 from 08.06.2021.

\section{Double Brackets for Open Arcs on Oriented Surfaces}

\subsection{A Review of Constructive Approach to Goldman Brackets}

For an oriented surface $\Sigma$ and a linear algebraic group $G$ we can associate the character variety
\begin{align*}
\mathrm{Hom}(\pi_1(\Sigma),G)\git G.
\end{align*}
It is well-known that the smooth locus of the character variety can be equipped with a Poisson structure \cite{AtiyahBott'1983, Goldman'1986, GuruprasadHuebschmannJeffreyWeinstein'1997}. Moreover, it was shown in \cite{Goldman'1986} that the Poisson bracket on the character variety can be computed from the Lie bracket on the vector space freely generated by homotopy equivalence classes of loops on $\Sigma$.

An elegant way to define the Lie bracket on, generally infinite dimensional, vector space of loops was suggested by G.~Massuyeau and V.~Turaev in \cite{MassuyeauTuraev'2014}.
To this end they have defined a double quasi Poisson bracket on the group algebra $A=\mathbf k\pi_1(\Sigma,p)$ of the fundamental group of a surface based at $p\in\partial\Sigma$.
A major technical advantage of the double bracket \cite{VandenBergh'2008} compared to the Lie bracket on the space of loops, is that a double bracket can be defined on a finite number of generators of $A$ and then be extended to $A\otimes A$ by some form of Leibniz identity. By construction, double brackets induce the Lie bracket on the cyclic space $A_\natural=A/[A,A]$.

The construction of G.~Massuyeau and V.~Turaev can be easily generalized to the case of a surface with several marked points on the boundary.
This was done in \cite{Arthamonov'2018} (see also \cite{ArthamonovRoubtsov}). This generalization is especially convenient in application to the
cluster algebras of tagged triangulations of 2d surfaces \cite{FominShapiroThurston'2008}, because it allows one to define mutations locally in terms of open arcs.
This idea was used in \cite{Arthamonov'2018} to define a noncommutative analogue of the spider move and in \cite{Ovenhouse'2020} in application to the noncommutative pentagram map.

\subsection{Double Quasi Poisson Brackets for Open Arcs}
\label{sec:DoubleBracketForOpenArcs}

Let $\Sigma$ be an oriented surface with nonempty boundary and $n>1$ marked points $p_1,\dots,p_n\in\partial\Sigma$ on the boundary.
Denote by $\mathcal C=\mathbf k\pi_1(\Sigma,p_1,\dots,p_n)$ the $k$-linear category generated by paths starting and terminating at
one of the marked points. Following the standard convention in topology we denote concatenation of (linear combinations of) paths
simply by juxtaposition\footnote{Note that concatenation of paths is commonly written in the opposite order to the composition of morphisms in $\mathcal C$, i.e. $fg=g\circ f$.}
\begin{align}
\mu:\quad \mathrm{Hom}(A,B)\otimes\mathrm{Hom}(B,C)\rightarrow \mathrm{Hom}(A,C),\qquad x \otimes y \mapsto xy.
\label{eq:ConcatenationLinear}
\end{align}
We extend concatenation (\ref{eq:ConcatenationLinear}) on arbitrary tensor products of componentwise composable paths as
\begin{align*}
(f\otimes g)(h\otimes m)=(fh\otimes gm).
\end{align*}
In addition, throughout the text we denote by $\tau$ the transposition of tensor factors $(f\otimes g)^\tau=g\otimes f$ for all $f,g\in\mathcal C$.

Together with a category $\mathcal C$ we can associate a \textit{cyclic space} $\mathcal C_\natural$, which is a $\mathbf k$-linear space of unbased loops
\begin{align}
\mathcal C_\natural=\left(\bigoplus_{X\in\mathrm{Obj}(\mathcal C)}\mathrm{Hom}(X,X)\right)\Big/[\mathcal C,\mathcal C].
\label{eq:Abelianization}
\end{align}
For a loop $f\in\mathcal C$ we denote by $\overline f\in\mathcal C_\natural$ its natural image in the cyclic space.\footnote{One can further assume that $\overline g=0$ for every $g\in\mathcal C$ which is not a loop. We do not need this convention in our text, however, it is worth noting that such convention allows one to define a map $\mathcal C\rightarrow\mathcal C_\natural,\;f\mapsto \overline f$ which is commonly referred to as ``universal trace''.}

Following the idea of double geometry suggested in \cite{Crawley-BoeveyEtingofGinzburg'2007,VandenBergh'2008} we define vector fields on $\mathcal C$ as double derivations.
\begin{definition}
Let $V,W\in\mathrm{Obj}(\mathcal C)$ be a pair of objects. We say that a family of $\mathbf k$-linear maps
\begin{align*}
\delta:\quad \mathrm{Hom}(A,B)\rightarrow\mathrm{Hom}(A,W)\otimes\mathrm{Hom}(V,B)\qquad\textrm{for all}\quad A,B\in\mathrm{Obj}(\mathcal C)
\end{align*}
is a $(V,W)$-vector field if it satisfies the following form of Leibnitz identity
\begin{align*}
\delta(fg)=(f\otimes{\mathbf 1}_V)\delta(g)+\delta(f)({\mathbf 1}_W\otimes g)\qquad\textrm{for all composable}\quad f,g\in\mathcal C.
\end{align*}
\label{def:VectorFields}
\end{definition}

\begin{example}
With every object $V\in\mathrm{Obj}(\mathcal C)$ we can associate a $(V,V)$-vector field, such that
\begin{equation}
\begin{aligned}
\partial_V:&\quad \mathrm{Hom}(A,B)\rightarrow \mathrm{Hom}(A,V)\otimes \mathrm{Hom}(V,B)\\[5pt]
\partial_V(f)=&\left\{\begin{array}{lll}
f\otimes \mathbf{1}_V-\mathbf{1}_V\otimes f,&A=V,&\quad B=V,\\[3pt]
-\mathbf{1}_V\otimes f,&A=V,&\quad B\neq V,\\[3pt]
f\otimes \mathbf{1}_V,&A\neq V,&\quad B=V,\\[3pt]
0,&A\neq V,&\quad B\neq V,
\end{array}\right.
\end{aligned}
\label{eq:UniderivationV}
\end{equation}
for all $A,B\in\mathrm{Obj}(\mathcal C)$ and $f\in\mathrm{Hom}(A,B)$. We refer to $\partial_V$ as uniderivation associated to an object $V\in\mathrm{Obj}(\mathcal C)$.
\end{example}

Collection of all vector fields defines a bifunctor of $\mathcal C$. The latter allows one to introduce a category $\mathcal V^{\mathcal C}$ generated by vector fields.\footnote{More details can be found in \cite{Arthamonov'2018,ArthamonovRoubtsov}.} We refer to $\mathcal V^{\mathcal C}$ as \textit{category of polyvector fields (on $\mathcal C$)}. Hereinafter, we denote the composition of polyvector fields by an asterisk $\ast$ in order to distinguish it from composition of paths. Note that $\mathcal V^{\mathcal C}$ comes equipped with the grading coming from the degree of a polyvector field.

\begin{remark}
An important particular case of Definition \ref{def:VectorFields} is when $\mathcal C$ has a single object \cite{Crawley-BoeveyEtingofGinzburg'2007,VandenBergh'2008}.
In other words, when $\mathcal C = \mathcal A$ is an associative algebra. In this case, the collection of all vector fields $Der(\mathcal A,\mathcal A\otimes\mathcal A)$
forms an $\mathcal A$-bimodule. The algebra of polyvector fields on $\mathcal A$ is nothing but the tensor algebra $T_{\mathcal A} Der(\mathcal A,\mathcal A\otimes\mathcal A)$
over $\mathcal A$ generated by vector fields.
\end{remark}

Similar to (\ref{eq:Abelianization}) we define the cyclic space $\mathcal V^{\mathcal C}_\natural$ consisting of ``traces'' of polyvector fields.
The cyclic space inherits grading by the degree of a polyvector field from $\mathcal V^{\mathcal C}$. It was shown in \cite{VandenBergh'2008,Arthamonov'2018}
that homogeneous components of $\mathcal V^{\mathcal C}_\natural$ are in one-to-one correspondence with polyderivations on $\mathcal C$.

\begin{definition}
We say that a collection of linear maps
\begin{align}
\db{,}:\quad \mathrm{Hom}(A,B)\otimes\mathrm{Hom}(C,D) \rightarrow \mathrm{Hom}(C,B)\otimes\mathrm{Hom}(A,D),\qquad\textrm{for all\quad $A,B,C,D\in\mathrm{ Obj}(\mathcal C)$}
\label{eq:DoubleBracketSourcesAndTargets}
\end{align}
is a \textit{double quasi Poisson bracket} on $\mathcal C$, if it satisfies the following properties:
\begin{subequations}
\begin{itemize}
\item \textbf{Skew-symmetry}
\begin{align}
\db{f,g}=-\left(\db{g,f}\right)^\tau,
\label{eq:DoubleBracketSkewSymmetry}
\end{align}
\item \textbf{Double Leibniz Identity}
\begin{align}
\db{fg,h}=&(\mathbf{1}_{s(h)}\otimes f)\db{g,h}+\db{f,h}(g\otimes \mathbf{1}_{t(h)}),\qquad\textrm{whenever $f,g$ are composable,}
\label{eq:DoubleLeibnizIdentityI}\\
\db{f,g h}= &(g\otimes\mathbf{1}_{s(f)})\db{f,h}+\db{f,g}(\mathbf{1}_{t(f)}\otimes h),\qquad\textrm{whenever $g,h$ are composable,}
\label{eq:DoubleLeibnizIdentityII}
\end{align}
here $s(f)$ and $t(f)$ stand for the source and target of $f\in\mathcal C$ respectively.
\item \textbf{Double Quasi Jacobi Identity}
\begin{align}
R_{1,2}R_{2,3}+R_{2,3}R_{3,1}+R_{3,1}R_{1,2}=\frac14\sum_{V\in\mathrm{Obj}(\mathcal C)}\overline{\partial_V\ast\partial_V\ast\partial_V}.
\label{eq:DoubleQuasiJacobiIdentity}
\end{align}
\end{itemize}
\end{subequations}
Here $R_{i,j}: (Mor\,\mathcal C)^{\otimes 3}\rightarrow(Mor\,\mathcal C)^{\otimes 3}$ stands for the double bracket between $i$'th and $j$'th tensor component:
\begin{equation*}
\begin{aligned}
R_{i,j}(f_1\otimes\dots\otimes f_3)=&f_1\otimes\dots\otimes\underbrace{\db{f_i, f_j}'}_i\otimes \dots\otimes\underbrace{\db{f_i, f_j}''}_j\otimes\dots\otimes f_3,
\end{aligned}
\end{equation*}
where we have used Sweedler notation $\db{f_i,f_j}=\db{f_i,f_j}'\otimes\db{f_i,f_j}''$ for the tensor components of the double bracket.

Each term of the form $\frac14\,\overline{\partial_V\ast\partial_V\ast\partial_V}$ on the r.h.s. of double quasi Jacobi identity (\ref{eq:DoubleQuasiJacobiIdentity}) is a triple derivation associated to a tri-vector field $\frac14\partial_V\ast\partial_V\ast\partial_V$.
\label{def:DoubleQuasiBracket}
\end{definition}

\begin{remark}
Definition \ref{def:DoubleQuasiBracket} is essentially equivalent to what is called a ``$B$-linear double quasi Poisson bracket'' in \cite{VandenBergh'2008}, where $B$ stands for the subalgebra generated by idempotents. For our purpose, however, it is very fruitful to keep track of sources and targets explicitly. Not only this approach is aesthetically more pleasing, but it also explains the meaning of the ansatz (\ref{eq:RMatrixAnsatz}) for the $r$-matrix. (See also Remark \ref{rem:TwistedCommutatorSourcesAndTargets}.)
\end{remark}

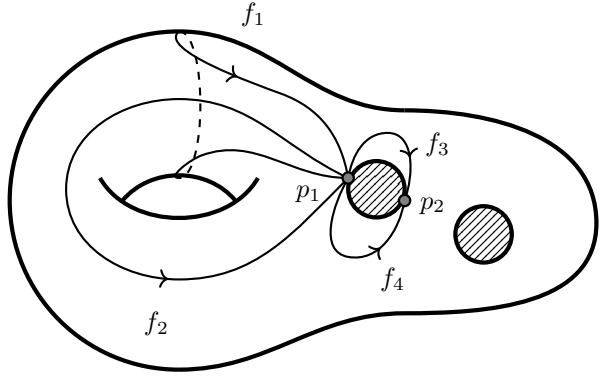
\begin{figure}
\begin{tikzpicture}[scale=1.5]
  \draw[thick, -->-] (1.5,0.2) to[out=155,in=0] (0,0.9) to[out=180,in=90] (-1,0.1) to[out=270,in=180] (0,-0.7) to[out=0,in=225] (1.5,0.2);
  \draw[thick,--<-] (1.5,0.2) to[out=100,in=-20] (0.9,0.95) to[out=160,in=220,looseness=0.6] (0,1.5);
  \draw[thick,dashed] (0,1.5) to[out=0,in=0,looseness=0.5] (0,0.2);
  \draw[thick] (0,0.2) to[out=180,in=200] (0.2,0.4) to[out=20,in=180] (1.5,0.2);
  \draw[thick,-->-] (1.5,0.2) to[out=240,in=170] (1.5,-0.5) to[out=350,in=260] (2,0);
  \draw[thick,-->-] (1.5,0.2) to[out=80,in=180] (1.85,0.6) to[out=0,in=70] (2,0);
  \fill [pattern= north east lines] (1.75,0.1) circle (0.25);
  \draw [ultra thick] (1.75,0.1) circle (0.25);
  \fill [pattern= north east lines] (2.7,-0.3) circle (0.25);
  \draw [ultra thick] (2.7,-0.3) circle (0.25);
  \fill (1.5,0.2) circle (0.06);
  \fill[gray] (1.5,0.2) circle (0.04);
  \fill (2,0) circle (0.06);
  \fill[gray] (2,0) circle (0.04);
  \draw[ultra thick] (0,1.5) to[out=0,in=180] (2,0.8);
  \draw[ultra thick] (0,1.5) arc (90:270:1.5);
  \draw[ultra thick] (0,-1.5) to[out=0,in=180] (2,-1) to[out=0,in=270] (3.7,-0.2) to[out=90,in=0] (2,0.8);
  \draw[ultra thick] (-0.5,0) to[out=50,in=130] (0.5,0);
  \draw[ultra thick] (-0.7,0.2) to[out=-60,in=240] (0.7,0.2);
  \draw (0.85,1.65) node[left] {$f_1$};
  \draw (-0.2,-0.9) node[below] {$f_2$};
  \draw (2.1,0.5) node[right] {$f_3$};
  \draw (1.7,-0.7) node[right] {$f_4$};
  \draw (1.15,-0.1) node[above] {$p_1$};
  \draw (2.25,-0.2) node[above] {$p_2$};
\end{tikzpicture}
\caption{Example of generating arcs for $\pi_1(\Sigma_{1,2},p_1,p_2)$.}
\label{fig:ExampleOfGeneratingArcs}
\end{figure}

In order to define the double quasi Poisson bracket on $\mathcal C$, pick a collection of open arcs $f_1,\dots,f_k$ which freely generate
$\mathcal C=\mathbf k\pi_1(\Sigma,p_1,\dots,p_n)$. This is possible, because $\partial \Sigma\neq\emptyset$.
Moreover, without loss of generality we can assume that generating arcs do not intersect anywhere except for the endpoints $p_1,\dots,p_n\in\partial\Sigma$.
As a result we obtain an oriented multigraph with vertex set $\{p_1,\dots,p_n\}$. Orientation of a surface defines the total order of half-edges adjacent to a given vertex. An example of such a collection of generating arcs for a torus with two boundary components and two marked points can be found on Figure \ref{fig:ExampleOfGeneratingArcs}.

Consider a marked point $p_l\in\partial\Sigma$ and denote by
\begin{align*}
S_l=\{x_1,\dots x_m\},
\end{align*}
the ordered set of half-edges adjacent to this vertex listed in the counterclockwise order according to the fixed orientation of the surface.
It will be convenient for us to label the half-edges by generators and their inverses, so $x_i=f_{n(i)}^{\epsilon(i)}$ with $\epsilon(i)=+1$
is for an outgoing half-edge and $\epsilon(i)=-1$ is for an incoming half-edge. The contribution from $p_l$ to the quasi Poisson bivector reads
\begin{subequations}
\begin{align}
P_{p_l}=\frac12\sum_{i<j}\big( x_i\ast\frac\partial{\partial x_i}\ast x_j\ast\frac\partial{\partial x_j}-x_j\ast\frac\partial{\partial x_j}\ast x_i\ast\frac\partial{\partial x_i}\big).
\label{eq:BivectorContribution}
\end{align}
Here $\frac\partial{\partial f_i}\in D_{t(f_i),s(f_i)}$ is a vector field on a category $\mathcal C$ defined on generators as
\begin{align*}
\frac\partial{\partial f_i}(f_j)=\left\{\begin{array}{ll}
1_{s(f_i)}\otimes 1_{t(f_i)},&i=j,\\
0,&i\neq j.
\end{array}\right.
\end{align*}
The derivation with respect to the inverses of generators is defined as
\begin{align*}
\frac\partial{\partial(f_i^{-1})}=-f_i\ast\frac\partial{\partial f_i}\ast f_i.
\end{align*}
Finally, the double quasi Poisson bracket on $\mathcal C$ is then computed by adding contributions from all the marked points on the boundary
\begin{align}
P_{\Sigma}=\sum_{l=1}^nP_{p_l},\qquad\qquad \db{,}=\overline{P_\Sigma},
\label{eq:DoubleBracketRibbonGraphSum}
\end{align}
\label{eq:DoubleBracketRibbonGraph}
\end{subequations}
where $\overline{P_\Sigma}$ stands for the biderivation associated to a noncommutative bivector $P_\Sigma$.

\subsection{Double Bracket and Standard Operations on Surfaces}

One of the most important properties of double brackets (\ref{eq:DoubleBracketRibbonGraph}), which becomes manifest in the approach of \cite{MassuyeauTuraev'2014}, is that the bracket on $\mathbf k\pi_1(\Sigma,p_1,\dots,p_n)$ does not depend on the choice of generating arcs.
Moreover, the double bracket behaves naturally with respect to the standard operations like adding/removing a marked point or gluing two surfaces together.
In particular, from the construction of \cite{MassuyeauTuraev'2014} we immediately get the following lemma.
\begin{lemma}[Adding/Removing a Marked Point]
    Let $\Sigma$ be an oriented surface with nonempty boundary and marked points $p_0,\dots,p_n\in\partial\Sigma$. Consider a full subcategory $\mathcal C'=\mathbf k\pi_1(\Sigma,p_1,\dots,p_n)$ and denote by $\imath:\mathcal C'\hookrightarrow\mathcal C$, the natural inclusion functor. Then, double brackets (\ref{eq:DoubleBracketRibbonGraph}) on $\mathcal C$ and $\mathcal C'$ are related as
    \begin{align*}
    \imath\left(\db{f,g}_{\mathcal C'}\right)=\db{\imath(f),\imath(g)}_{\mathcal C}\qquad\qquad\textrm{for all $f,g\in Mor\,\mathcal C'$.}
    \end{align*}
    \label{lemm:AddingRemovingPoint}
\end{lemma}
In what follows we always omit natural inclusion functors $\imath:\mathcal C'\hookrightarrow\mathcal C$, whenever the inclusion is clear from the context.

The standard procedure for combining two quasi Poisson algebras into a new one is known as \textit{fusion} \cite{AlekseevKosmann-SchwarzbachMeinrenken'2002}.
In our case, this corresponds to gluing two surfaces $\Sigma_1$ and $\Sigma_2$ along the (segment of the) boundary, matching the marked point as shown on Figure \ref{fig:Fusion}. Let $\Sigma_1,\Sigma_2$ be oriented surfaces with nonempty boundary together with a choice of marked points $p_0,\dots, p_m\in\partial \Sigma_1$ and $q_0,\dots,q_n\in\partial\Sigma_2$. Consider a surface $\Sigma=\Sigma_1\#_{L}\Sigma_2$ obtained as gluing of $\Sigma_1$ and $\Sigma_2$ along the segment of the boundary $L$ matching the marked points $p_0\in\partial\Sigma_1$ and $q_0\in\partial\Sigma_2$ as shown on Figure \ref{fig:Fusion}. We assume that no other marked points appear on the glued segments. The choice of marked points on both halves defines a choice of $m+n-1$ marked points $\phi_0,p_1,\dots,p_n,q_1,\dots,q_m\in\partial\Sigma$, where $\phi_0$ stands for the common image of $p_0$ and $q_0$ in $\partial\Sigma$. By analogy with notations used in \cite{AlekseevKosmann-SchwarzbachMeinrenken'2002} we denote the resulting $\mathbf k$-linear category of paths on $\Sigma$ as $\mathcal C=\mathcal C_1\prescript{}{p_0}{\oast}^{}_{q_0}\mathcal C_2,$ where $\oast$ stands for the fusion symbol.

Note that an arbitrary polyvector field $\delta\in\mathcal V^{\mathcal C_1}$ on $\mathcal C_1$ (resp. $\mathcal C_2$) can be extended to a polyvector field on $\mathcal C$ by simply declaring that $\delta$ acts trivially on generators of $\mathcal C_2$ (res. $\mathcal C_1$). To avoid cumbersome notations we denote the induced (poly)vector field by the same symbol $\delta\in\mathcal V^{\mathcal C}$. In particular, the bivectors $P_{\Sigma_1},P_{\Sigma_2}$ defined in (\ref{eq:DoubleBracketRibbonGraph}) induce bivectors on $\mathcal C$. At the same time, uniderivations $\partial_{p_0}\in\mathcal V^{\mathcal C_1}$ and $\partial_{q_0}\in\mathcal V^{\mathcal C_2}$ induce vector fields on $\mathcal C$ of degree one.

The following lemma is originally due to \cite{AlekseevKosmann-SchwarzbachMeinrenken'2002,AlekseevMalkinMeinrenken'1998} (see also \S6 in \cite{VandenBergh'2008} for equivalent operation on quasi Poisson brackets for quiver path algebras, and \cite{Nie'2013} for a good recent review of fusion in the context of surface brackets).
\begin{figure}
\centering
\begin {tikzpicture}
    \fill[blue, opacity=0.2] (-3,0) -- (0,0) -- (0,3) -- (-3,3) -- cycle;
    \fill[red,  opacity=0.2] (3,0) -- (0,0) -- (0,3) -- (3,3) -- cycle;

    \draw [red, -latex] (0,0) to [out=30, in=180] (3,1);
    \draw [red, -latex] (0,0) to [out=45, in=180] (3,2);
    \draw [red, -latex] (0,0) to [out=60, in=-90] (1,3);

    \draw [blue, -latex] (0,0) to [out=120, in=-90] (-1,3);
    \draw [blue, -latex] (0,0) to [out=135, in=-90] (-2,3);
    \draw [blue, -latex] (0,0) to [out=150, in=0] (-3,1);

    \draw (-3.5,1.5) node {\color{blue} $\Sigma_1$};
    \draw (3.5,1.5)  node {\color{red} $\Sigma_2$};

    \draw (2,0.6) node {\color{red} $y_1$};
    \draw (2,1.5) node {\color{red} $y_2$};
    \draw (1.5,2.2) node {\color{red} $\ddots$};
    \draw (0.7,2.5) node {\color{red} $y_n$};

    \draw (-2,0.6)   node {\color{blue} $x_m$};
    \draw (-2.2,1.5) node {\color{blue} \reflectbox{$\ddots$}};
    \draw (-1.7,2.5) node {\color{blue} $x_2$};
    \draw (-0.7,2.5) node {\color{blue} $x_1$};

    \draw (-3,0) -- (3,0);
    \draw [dashed] (0,0) -- (0,3);

    \draw[fill=black] (0,0) circle (0.05);
    \draw (0,0) node [below] {$\phi_0$};
\end {tikzpicture}
\caption{Fusion of two quasi Poisson structures associated to gluing of surfaces}
\label{fig:Fusion}
\end{figure}
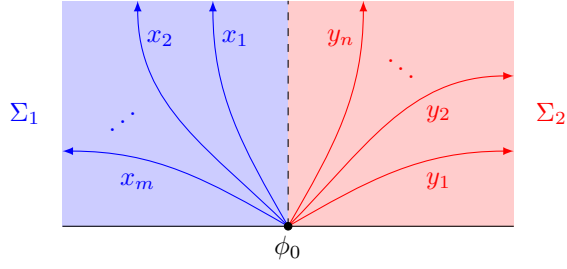

\begin{lemma}[Fusion]
Let $\Sigma=\Sigma_1\#_{L}\Sigma_2$ be a surface obtained by gluing $\Sigma_1,\Sigma_2$ along the segment of the boundary matching $p_0\in\partial\Sigma_1$ with $q_0\in\partial\Sigma_2$ as shown on Figure \ref{fig:Fusion}. The quasi Poisson bivectors (\ref{eq:DoubleBracketRibbonGraph}) associated to the three surfaces are related as follows:
\begin{align}
P_{\Sigma}-P_{\Sigma_1}-P_{\Sigma_2}= \frac12\left(\partial_{p_0}\ast\partial_{q_0}-\partial_{q_0}\ast\partial_{p_0}\right).
\label{eq:Fusion}
\end{align}
\label{lemm:Fusion}
\end{lemma}
\begin{proof}
Choose two collections of arcs which generate freely $\mathcal C_1=\mathbf k\pi_1(\Sigma_1,p_0,\dots,p_n)$ and $\mathcal C_2=\mathbf k\pi_1(\Sigma_2,q_0,\dots,q_n)$ respectively.
The contributions from $P_{p_1},\dots,P_{p_m}$ and $P_{q_1},\dots,P_{q_n}$ cancel on the r.h.s. of (\ref{eq:Fusion}). On the other hand, let $S_{p_0}=\{x_1,\dots,x_m\}$
and $S_{q_0}=\{y_1,\dots,y_n\}$ denote the ordered sets of half-edges adjacent to $p_0$ and $q_0$ respectively (see Figure \ref{fig:Fusion}).
As a corollary, the ordered set of half-edges adjacent to their common image in $\Sigma$ reads $S_{\phi_0} = \left\{y_1,\dots,y_n,x_1,\dots,x_m\right\}$.
Hence, the difference of the associated bivectors (\ref{eq:BivectorContribution}) has the following form
\begin{align*}
P_{\phi_0}^{\Sigma}-P_{p_0}^{\Sigma_1}-P_{q_0}^{\Sigma_2}= &\frac12\sum_{i=1}^{m}\sum_{j=1}^n\big(x_i\ast\frac\partial{\partial x_i}\ast y_j\ast\frac\partial{\partial y_j} -y_j\ast\frac\partial{\partial y_j}\ast x_i\ast\frac\partial{\partial x_i}\big)\\
=&\frac12\left(\partial_{p_0}\ast\partial_{q_0}-\partial_{q_0}\ast\partial_{p_0}\right).
\end{align*}
\end{proof}

\begin{remark}
Note that the r.h.s. of (\ref{eq:Fusion}) is not symmetric with respect to exchange $\Sigma_1\leftrightarrow\Sigma_2$. As a result, fusion of the two quasi Poisson categories is not symmetric either
\begin{align*}
\mathcal C_1\prescript{}{p_0}{\oast}^{}_{q_0}\mathcal C_2\neq \mathcal C_2\prescript{}{q_0}{\oast}^{}_{p_0}\mathcal C_1.
\end{align*}
This exactly corresponds to the two ways of gluing $\Sigma_1,\Sigma_2$ matching the marked point as shown on Figure \ref{fig:FusionNonSymmetric}.
\begin{figure}
\begin {tikzpicture}
    \fill[red, opacity=0.3]  (1,0) -- (2,1) -- (2,-1) -- cycle;
    \fill[blue, opacity=0.3] (0,0) -- (-1,1) -- (-1,-1) -- cycle;

    \draw [dashed] (0,0) -- (-1,1);
    \draw [dashed] (1,0) -- (2,1);

    \draw [thick, dotted] (0,0) -- (-1,-1);
    \draw [thick, dotted] (1,0) -- (2,-1);

    \draw [fill=black] (0,0) circle (0.05);
    \draw [fill=black] (1,0) circle (0.05);

    \draw (-0.6,0) node {$\Sigma_1$};
    \draw (1.7,0)  node {$\Sigma_2$};

    \draw (0,0.4) node {$p_0$};
    \draw (1,0.4) node {$q_0$};


    \draw [-latex] (3,1) -- (4,1.5);
    \draw [-latex] (3,-1) -- (4,-1.5);

    \draw (3.5,2) node {$\mathcal{C}_1 \prescript{}{p_0}{\oast}^{}_{q_0} \mathcal{C}_2$};
    \draw (3.5,-2) node {$\mathcal{C}_2 \prescript{}{q_0}{\oast}^{}_{p_0} \mathcal{C}_1$};


    \fill [blue, opacity=0.3] (5,1) -- (6,1) -- (6,2) -- (5,2) -- cycle;
    \fill [red, opacity=0.3]  (6,1) -- (7,1) -- (7,2) -- (6,2) -- cycle;

    \draw [thick, dotted] (5,1) -- (7,1);
    \draw [dashed] (6,1) -- (6,2);

    \draw [fill=black] (6,1) circle (0.05);


    \fill [blue, opacity=0.3] (5,-1) -- (6,-1) -- (6,-2) -- (5,-2) -- cycle;
    \fill [red, opacity=0.3]  (6,-1) -- (7,-1) -- (7,-2) -- (6,-2) -- cycle;

    \draw [dashed] (5,-1) -- (7,-1);
    \draw [thick, dotted] (6,-1) -- (6,-2);

    \draw [fill=black] (6,-1) circle (0.05);

\end {tikzpicture}
\caption{Two ways of gluing $\Sigma_1$ with $\Sigma_2$ matching the marked point.}
\label{fig:FusionNonSymmetric}
\end{figure}
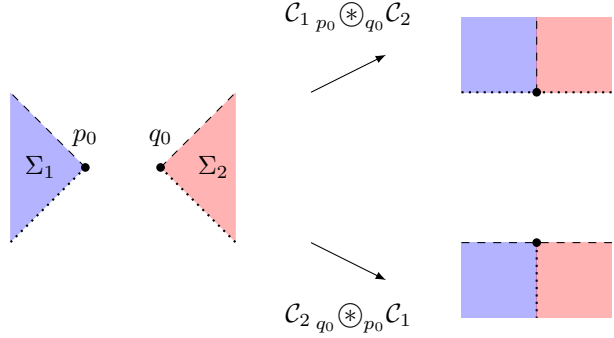
\end{remark}

On the other hand, the r.h.s. of (\ref{eq:Fusion}) does not contribute to the quasi Poisson bracket on full subcategory
\begin{align*}
\mathcal C'=\mathbf k\pi_1(\Sigma,p_1,\dots,p_m,q_1,\dots,q_n),
\end{align*}
where we have forgotten the vertex $\phi_0\in\partial\Sigma$. This leads to the following
\begin{proposition}
Let $\Sigma_1,\Sigma_2$ be oriented surfaces with nonempty boundary together with a choice of marked points $p_0,\dots,p_n\in\partial \Sigma_1$ and $q_0,\dots,q_n\in\partial\Sigma_2$. Consider a surface $\Sigma=\Sigma_1\#_L\Sigma_2$ obtained as gluing $\Sigma_1$ and $\Sigma_2$ along the (segment of) the boundary component containing $p_0,q_0$.
Denote by $\mathcal C'=\mathbf k\pi_1(\Sigma,p_1,\dots,p_m,q_1,\dots,q_n)$ the full subcategory of $\mathcal C=\mathbf k\pi_1(\Sigma,\phi_0,p_1,\dots,p_m,q_1,\dots,q_n)$. Then
\begin{align}
\left(P_{\Sigma}-P_{\Sigma_1}-P_{\Sigma_2}\right)\big|_{\mathcal C'}=0.
\label{eq:GluingBivector}
\end{align}
In particular, $P_{\Sigma_1}+P_{\Sigma_2}$ preserves $\mathcal C'$, hence reduces to a bivector on a subcategory.
\label{prop:GluingBivector}
\end{proposition}
\begin{proof}
Note that on the right hand side of the fusion equation (\ref{eq:Fusion}) we have a bivector which is constructed out of two uniderivations $\partial_{p_0}$ and $\partial_{q_0}$. We claim that this bivector acts trivially on $\mathcal C'$. Note that because the double bracket on $\mathcal C'$ does not depend on the choice of generating arcs, it is enough for us to prove the above claim for any choice of generators. To this end, let $x_1,\dots, x_M$ be generators for $\mathcal C_1'=\mathbf k\pi_1(\Sigma_1,p_1,\dots,p_m)$, while $y_1,\dots,y_N$ be generators for $\mathcal C_2'=\mathbf k\pi_1(\Sigma_2,q_1,\dots,q_n)$. From (\ref{eq:UniderivationV}) we know that both $\partial_{p_0}$ and $\partial_{q_0}$ act trivially on $\mathcal C_1'$ and $\mathcal C_2'$, as a corollary
\begin{align}
(\partial_{p_0}\ast\partial_{q_0}-\partial_{q_0}\ast\partial_{p_0}) (x_j\otimes\rule{0.3cm}{0.15mm}) =(\partial_{p_0}\ast\partial_{q_0}-\partial_{q_0}\ast\partial_{p_0}) (y_j\otimes\rule{0.3cm}{0.15mm})=0.
\label{eq:FusionRHSTrivialActionDisjoint}
\end{align}
On the other hand, consider $f=x_0 y_0$, where $x_0$ is an arbitrary arc from $p_1$ to $p_0$, while $y_0$ is an arbitrary arc from $q_0$ to $q_1$. By (\ref{eq:UniderivationV}) we have
\begin{align*}
\partial_{p_0}f=&\partial_{p_0}(x_0 y_0)=(\partial_{p_0}x_0)(\mathbf 1_{pq_0}\otimes y_0)
\stackrel{(\ref{eq:UniderivationV})}{=}(x_0\otimes\mathbf 1_{pq_0})(\mathbf 1_{pq_0}\otimes y_0)=(x_0\otimes y_0),\\[5pt]
\partial_{q_0}f=&\partial_{q_0}(x_0 y_0)= (x_0\otimes\mathbf 1_{pq_0})(\partial_{q_0}y_0)\stackrel{(\ref{eq:UniderivationV})}{=} -(x_0\otimes\mathbf 1_{pq_0})(\mathbf 1_{pq_0}\otimes y_0)=-(x_0\otimes y_0).
\end{align*}
As a result we have
\begin{align}
(\partial_{p_0}\ast\partial_{q_0}-\partial_{q_0}\ast\partial_{p_0})(f\otimes f)=x_0\otimes x_0 y_0\otimes y_0-x_0\otimes x_0 y_0\otimes y_0=0.
\label{eq:FusionRHSTrivialActionExtraMorphism}
\end{align}
To finalize the proof note that $f$ together with $x_1,\dots, x_M$ and $y_1,\dots,y_N$ generate $\mathcal C'$.
From (\ref{eq:FusionRHSTrivialActionDisjoint}) and (\ref{eq:FusionRHSTrivialActionExtraMorphism}) we conclude that the right hand side of the (\ref{eq:Fusion}) acts trivially on $\mathcal C'$.
\end{proof}

Analogues of Lemma \ref{lemm:Fusion} and Proposition \ref{prop:GluingBivector} hold for gluing two different segments of the boundary of the same surface matching the marked points. We omit the proof here, because it involves more cumbersome notation, yet, is essentially identical to the two-component case.

Note that Lemma \ref{lemm:AddingRemovingPoint} together with formula (\ref{eq:GluingBivector}) provides a constructive way of calculating the double bracket on a surface glued from many elementary pieces. Our calculations in Section \ref{sec:PlanarAndCylindricalNetworks} will be based on the following
\begin{figure}
\begin {tikzpicture}
    \fill [gray, opacity=0.3] (0,-1) -- (0,1) -- (-1,1) -- (-1,-1) -- cycle;
    \draw (0,-1) -- (0,1);

    \draw (-0.5,0.7) node {$\Sigma$};

    \draw [fill=black] (0,0) circle (0.05);

    \draw (0.25,0) node {\small $p_1$};

    \draw [-latex] (0,0) -- (-1,-0.3);
    \draw [-latex] (0,0) -- (-1,0.1);
    \draw [-latex] (0,0) -- (-1,0.4);

    \fill [gray, opacity=0.3] (1,-1) -- (1,1) -- (2,1) -- (2,-1) -- cycle;
    \draw (1,-1) -- (1,1);

    \draw (1.5,0.7) node {$\Sigma$};

    \draw (1,0) [fill=black] circle (0.05);

    \draw (0.75,0) node {\small $p_2$};

    \draw [-latex] (1,0) -- (2,0.2);
    \draw [-latex] (1,0) -- (2,-0.2);


    \draw [-latex] (2.5,0) -- (3.5,0);


    \fill [gray, opacity=0.3] (4,-1) -- (6,-1) -- (6,1) -- (4,1) -- cycle;

    \draw [dashed] (5,-1) -- (5,1);

    \draw (4.5,0.7) node {$\Sigma'$};

    \draw [fill=black] (5,0) circle (0.05);

    \draw (5.3,0.25) node {\small $p_{1,2}$};

    \draw [-latex] (5,0) -- (4,-0.3);
    \draw [-latex] (5,0) -- (4,0.1);
    \draw [-latex] (5,0) -- (4,0.4);
    \draw [-latex] (5,0) -- (6,0.2);
    \draw [-latex] (5,0) -- (6,-0.2);

\end {tikzpicture}
\caption{Gluing matching internal marked point.}
\label{fig:GluingInternal}
\end{figure}
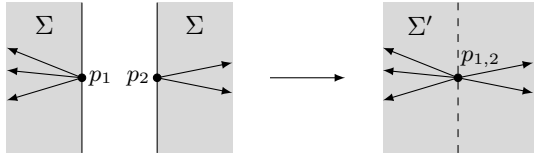

\begin{corollary}
Let $\Sigma$ be an oriented surface with arbitrary number of connected components.\footnote{For our purpose we actually need only one or two connected components.} Assume further that $\Sigma$ has nonempty boundary with $m\geqslant3$ marked points $p_1,\dots,p_m\in\partial\Sigma$ and consider a surface $\Sigma'$ obtained by gluing a segment of $\partial\Sigma$ near $p_1$ and $p_2$ matching the two marked points as shown on Figure \ref{fig:GluingInternal}.

For any choice of generating arcs $x_1,\dots,x_k\in\mathbf\pi_1(\Sigma,p_1,\dots,p_m)$ and  $y_1,\dots,y_l\in\mathbf k\pi_1(\Sigma',p_3,\dots,p_m)$, where
\begin{align*}
y_i=F_i(x_1,\dots,x_k),\qquad 1\leqslant i\leqslant l
\end{align*}
we have
\begin{align}
\db{y_i,y_j}_{\Sigma'}=\db{F_i(x_1,\dots,x_k),F_j(x_1,\dots,x_k)}_{\Sigma},\qquad 1\leqslant i,j\leqslant l.
\label{eq:GluingCommute}
\end{align}
\label{cor:GluingCommute}
\end{corollary}

In other words, we can think of the common image $p_{1,2}\in\Sigma'$ of boundary points $p_1,p_2\in\partial\Sigma$ as an internal point located on the fixed cut of $\Sigma'$ and still define the biderivation
\begin{align*}
\db{-,-}_{\Sigma_\#'}:\quad\mathcal C_{\#}'\otimes \mathcal C_{\#}'\rightarrow \mathcal C_{\#}'\otimes\mathcal C_{\#}',\qquad\textrm{where}\quad\mathcal C_{\#}'=\mathbf k\pi_1(\Sigma',p_{1,2},p_3,\dots,p_m).
\end{align*}
by simply computing the double bracket between generating arcs on $\Sigma$. This biderivation coincides with the usual double bracket on a subcategory $\mathcal C'=\mathbf k\pi_1(\Sigma',p_3,\dots,p_m)$
\begin{align*}
\db{-,-}_{\Sigma_\#'}\Big|_{\mathcal C'}=\db{-,-}_{\Sigma'}.
\end{align*}

\begin{remark}
Note, however, that $\db{-,-}_{\Sigma_\#'}$ records additional information about the cut. In particular, the trivector on the r.h.s. of the double quasi Jacobi identity for $\db{-,-}_{\Sigma_\#'}$ contains contributions from $p_1$ and $p_2$ respectively, which remarkably cancel each other on a subcategory $\mathcal C'$.
\end{remark}

\section{Planar and Cylindrical Networks}

\label{sec:PlanarAndCylindricalNetworks}

We now define our main objects of interest, which are certain graphs embedded on surfaces. This first
definition is due to Postnikov.

\subsection {Perfect Planar Networks and Boundary Measurements}

\begin {definition} \cite{Postnikov'2006}
    A \emph{perfect planar network} is a directed graph embedded in a disk, such that
    \begin {itemize}
        \item vertices on the boundary are univalent
        \item internal vertices are trivalent
        \item internal vertices are neither sources nor sinks
    \end {itemize}
\end {definition}

The definition implies that boundary vertices are necessarily either sources or sinks, and that internal vertices come in one of
two types: they have either a unique incoming edge or a unique outgoing edge. We picture the former as white vertices, and the latter as black.

In this paper we only consider networks with sources and sinks on the boundary which are separated from each other.
This means that all sources occur consecutively and all sinks occur consecutively. We may then picture the disk in such a way that the sources all occur on
the left side, and all sinks on the right. Suppose there are $m$ sources and $n$ sinks. We will use the convention of labelling the sources
$1,\dots,m$ in counter-clockwise order, and the sinks $1,\dots,n$ in clockwise order. This way, the indices increase from top to bottom.
An example is shown in Figure \ref{fig:perfect_planar_network}.

\begin {figure}
    \begin {tikzpicture}
        \coordinate (i1) at (-1,1);
        \coordinate (i2) at (-1,-1);
        \coordinate (i3) at (1,-1);
        \coordinate (i4) at (1,1);

        \coordinate (b1) at (-2.29,1);
        \coordinate (b2) at (-2.29,-1);
        \coordinate (b3) at (2.29,-1);
        \coordinate (b4) at (2.29,1);

        \draw [fill=black] (i1) circle (0.05cm);
        \draw [fill=white] (i2) circle (0.05cm);
        \draw [fill=black] (i3) circle (0.05cm);
        \draw [fill=white] (i4) circle (0.05cm);

        \draw [-latex] (b1) -- ($(i1) + (-0.05,0)$);
        \draw [-latex] (b2) -- ($(i2) + (-0.05,0)$);
        \draw [-latex] ($(i3) + (0.05,0)$) -- (b3);
        \draw [-latex] ($(i4) + (0.05,0)$) -- (b4);

        \draw [-latex] ($(i1) + (0.05,0)$) -- ($(i4) + (-0.05,0)$);
        \draw [-latex] ($(i2) + (0,0.05)$) -- ($(i1) + (0,-0.05)$);
        \draw [-latex] ($(i2) + (0.05,0)$) -- ($(i3) + (-0.05,0)$);
        \draw [-latex] ($(i4) + (0,-0.05)$) -- ($(i3) + (0,0.05)$);

        \draw [dashed] (0,0) circle (2.5);

        \draw ($(b1) + (-0.5,0)$) node {$1$};
        \draw ($(b2) + (-0.5,0)$) node {$2$};
        \draw ($(b3) + (0.5,0)$) node {$2$};
        \draw ($(b4) + (0.5,0)$) node {$1$};

        \draw (0,1.3) node {$d$};
        \draw (1.3,0) node {$c$};
        \draw (-1.4,0) node {$a$};
        \draw (0,-1.4) node {$b$};
    \end {tikzpicture}
    \caption {A perfect planar network}
    \label {fig:perfect_planar_network}
\end {figure}
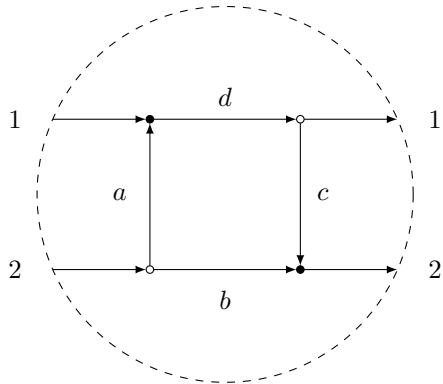

A path in $\Gamma$ can be written as a word in the generators of $\mathcal{C}$. If $p$ is a path, then we write $\mathrm{wt}(p)$
for the \emph{weight} of $p$, which is simply the corresponding word in the generators of $\mathcal{C}$.

\begin {definition}
    Suppose $\Gamma$ is an acyclic perfect planar network with $m$ sources and $n$ sinks on the boundary, which are separated from each other. We define an
    $m$-by-$n$ matrix with entries in $\mathcal{C}$, called the \emph{boundary measurement matrix}, and denoted $B_\Gamma$,
    as follows. The entry in position $i,j$ is the weight generating function for paths between source $i$ and sink $j$:
    \[ b_{ij} = \sum_{p \colon i \to j} \mathrm{wt}(p) \]
\end {definition}

\begin {remark}
    The definition for the boundary measurement matrix given here differs slightly from the one in \cite{Postnikov'2006}.
    Postnikov defined an $m \times (m+n)$ matrix where the extra columns indexed by the sources form the identity matrix,
    and also there were signs attached to each $b_{ij}$ so that all $m \times m$ minors are subtraction-free.
\end {remark}

For example, the network pictured in Figure \ref{fig:perfect_planar_network} has boundary measurement matrix
\[ B_\Gamma = \begin{pmatrix} d & dc \\ ad & b+adc \end{pmatrix} \]
The boundary measurement matrices behave nicely with respect to the gluing of disks matching sources and sinks. The following is well known.

\begin {proposition} \label{prop:bm_matrix_concatenation}
    Suppose $\Gamma_1$ and $\Gamma_2$ are both perfect planar networks with separated sources and sinks. Suppose $\Gamma_1$ has $m$ sources and $k$ sinks, and $\Gamma_2$
    has $k$ sources and $n$ sinks. Consider the network $\Gamma_3 = \Gamma_1 \# \Gamma_2$, which is obtained by glueing the sinks
    of $\Gamma_1$ to the sources of $\Gamma_2$. Then
    \[ B_{\Gamma_3} = B_{\Gamma_1} B_{\Gamma_2} \]
\end {proposition}

\subsection {Ribbon Graphs}

Given a perfect planar network $\Gamma$, we construct the \emph{ribbon graph} (also called \emph{fat graph}), which we denote
by $\Sigma_\Gamma$. It is a surface which is homotopy equivalent to $\Gamma$ formed by thickening the graph. More precisely,
there is a disk for each vertex of $\Gamma$, and a rectangle (or ribbon) for each edge, and the ribbons are glued to the disks
according to the cyclic order determined by $\Gamma$'s embedding in the plane. In such a case, $\Gamma$ is said to be a \emph{spine}
of the corresponding surface $\Sigma_\Gamma$. For each vertex of $\Gamma$, we put a marked point on the boundary of $\Sigma_\Gamma$.
At white vertices, the unique incoming edge is \emph{last} in the counter-clockwise order, and at black vertices, the unique outgoing
edge is last in the counter-clockwise order. This is depicted in Figure \ref{fig:ribbon_graph_pieces}.

\begin {figure}
\centering
\begin {tikzpicture} [scale=2.0]
    \draw [dashed] (0,0) circle (1.0);

    \draw (-0.98,-0.2) to [out=11, in=108.5] (0.32,-0.95);
    \draw [rotate=120] (-0.98,-0.2) to [out=11, in=108.5] (0.32,-0.95);
    \draw [rotate=-120] (-0.98,-0.2) to [out=11, in=108.5] (0.32,-0.95);

    \draw [fill=white] ($0.37*(-0.5,-0.866)$) circle (0.05);

    \draw [-latex] (-1,0) to [out=0, in=130] ($0.37*(-0.5,-0.866) + 0.04*(-1,1)$);
    \draw [-latex] ($0.32*(-0.5,-0.866)$) -- (0.5,0.866);
    \draw [-latex] ($0.37*(-0.5,-0.866) + 0.05*(1,0)$) to [out=20, in=120] (0.5,-0.866);


    \draw [dashed] (3,0) circle (1.0);

    \draw [shift={(3,0)}, rotate=60]  (-0.98,-0.2) to [out=11, in=108.5] (0.32,-0.95);
    \draw [shift={(3,0)}, rotate=180] (-0.98,-0.2) to [out=11, in=108.5] (0.32,-0.95);
    \draw [shift={(3,0)}, rotate=-60] (-0.98,-0.2) to [out=11, in=108.5] (0.32,-0.95);

    \draw [fill=black] ($(3,0) + 0.37*(0.5,0.866)$) circle (0.05);

    \draw [-latex, shift={(3,0)}] (-0.5,0.866) to [out=-60, in=-180] ($0.37*(0.5,0.866) + 0.04*(-1,-1)$);
    \draw [-latex, shift={(3,0)}] (-0.5,-0.866) -- ($0.32*(0.5,0.866)$);
    \draw [-latex, shift={(3,0)}] ($0.37*(0.5,0.866) + 0.05*(0,-1)$) to [out=-60, in=180] (1,0);
\end {tikzpicture}
\caption {Building blocks of the ribbon graph $\Sigma_\Gamma$}
\label {fig:ribbon_graph_pieces}
\end {figure}
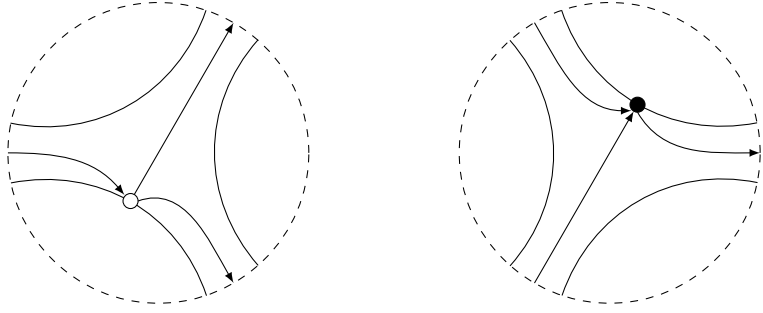

As described in the previous section, this gives a double bracket on the category
$\mathcal{C}$, consisting of paths in $\Gamma$. The set of edges in $\Gamma$ is a generating set for $\mathcal{C}$. In terms of these generators, the double bracket given in Equation \ref{eq:DoubleBracketRibbonGraph}
can be described locally at each vertex of $\Gamma$ by the following.

Let $f$ and $g$ be two directed edges in $\Gamma$ which are incident to a common vertex $V$, with $f < g$ in the total order. The contribution to the double bracket (\ref{eq:DoubleBracketRibbonGraph}) from the vertex $V$ has the following form
\begin {itemize}
    \item If $s(f) = s(g)$, then $\db{f,g} = \frac{1}{2} \, f \otimes g$
    \item If $t(f) = t(g)$, then $\db{f,g} = \frac{1}{2} \, g \otimes f$
    \item If $t(f) = s(g)$, then $\db{f,g} = -\frac{1}{2} \, \mathbf{1}_V \otimes fg$
\end {itemize}
The double bracket on $\mathcal{C}$ can then be computed by adding up contributions from all internal vertices of $\Gamma$.

\subsection {A Family of Double Brackets}

In addition to the double bracket described in the previous section, there is a natural family of double brackets
depending on $6$ parameters. It is a noncommutative generalization of the $6$-parameter family of Poisson brackets
from \cite{GSV'2009}, and was described in \cite{Ovenhouse'2019}.

Consider elementary building blocks of planar networks depicted on Figure \ref{fig:ribbon_graph_pieces}. For each of the building blocks we have three different choices for boundary component where we can put white (respectively black) vertex, as shown on Figure \ref{fig:ThreeChoices}. On the other hand, by Lemma \ref{lemm:AddingRemovingPoint}, each of the three choices gives rise to one and the same bracket on the subcategory $\mathbf k\pi_1(\Sigma_W,p_1,p_2,p_3)$ of paths starting/terminating at $p_1,p_2,p_3$. Indeed, for the case shown on Figure \ref{fig:ThreeChoices} we have
\begin{align*}
\mathbf k\pi_1(\Sigma,p_1,p_2,p_3)=&\mathbf k\langle l_{32},l_{31}\rangle,\qquad l_{31}=x_3x_1,\quad l_{32}=x_3x_2,\\
\db{l_{31},l_{32}}_a=&\db{l_{31},l_{32}}_b=\db{l_{31},l_{32}}_c=\frac12l_{31}\otimes l_{32}.
\end{align*}

This means that we may remove the white vertex from the boundary, and instead think of this local picture as having two paths
which enter at $p_3$, and exit at $p_1$ and $p_2$. All three choices of bracket above, corresponding to different placements
of the white vertex at different points on the boundary, induce the same bracket in this new picture. In this way, we may remove all internal vertices, so that the only marked boundary points on our ribbon graph correspond to the
vertices of $\Gamma$ on the boundary of the disk.

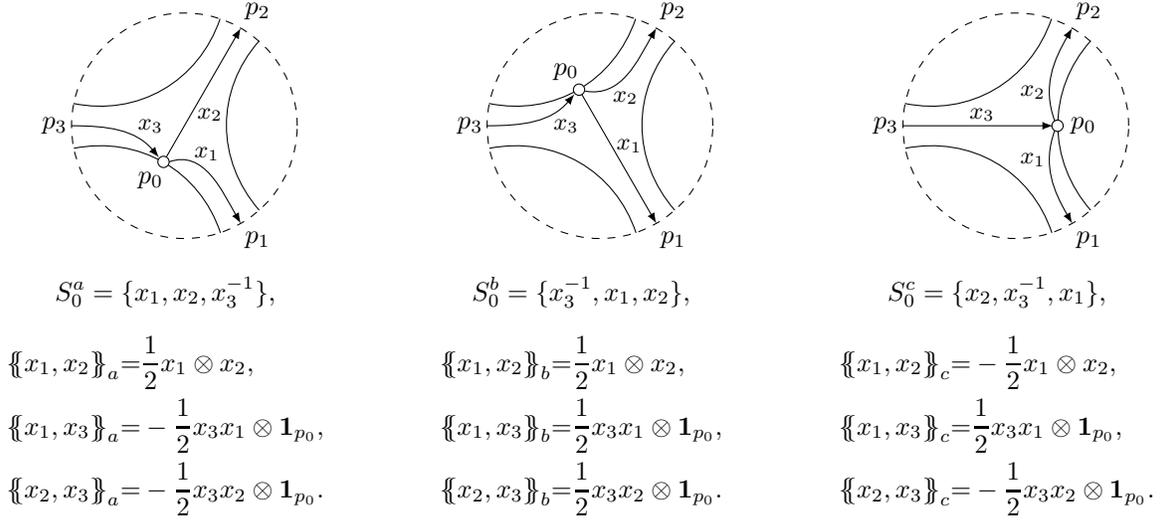
\begin{figure}
\begin{subfigure}{0.3\linewidth}
\centering
\begin {tikzpicture} [scale=1.5]
    \draw [dashed] (0,0) circle (1.0);

    \draw               (-0.98,-0.2) to [out=11, in=108.5] (0.32,-0.95);
    \draw [rotate=120]  (-0.98,-0.2) to [out=11, in=108.5] (0.32,-0.95);
    \draw [rotate=-120] (-0.98,-0.2) to [out=11, in=108.5] (0.32,-0.95);

    \draw [fill=white] ($0.37*(-0.5,-0.866)$) circle (0.05);

    \draw [-latex] (-1,0) to [out=0, in=130] ($0.37*(-0.5,-0.866) + 0.04*(-1,1)$);
    \draw [-latex] ($0.32*(-0.5,-0.866)$) -- (0.5,0.866);
    \draw [-latex] ($0.37*(-0.5,-0.866) + 0.05*(1,0)$) to [out=20, in=120] (0.5,-0.866);

    \draw (0.65,-1.016) node {$p_1$};
    \draw (0.65,1.016)  node {$p_2$};
    \draw (-1.15,0)     node {$p_3$};

    \draw (-0.3,-0.483) node {$p_0$};

    \draw (0.2,-0.25) node {\small $x_1$};
    \draw (0.225,0.1) node {\small $x_2$};
    \draw (-0.3,0)    node {\small $x_3$};
\end {tikzpicture}
\begin{align*}
S_0^a=\{x_1,x_2,x_3^{-1}\},
\end{align*}
\begin{align*}
\db{x_1,x_2}_a=&\frac12x_1\otimes x_2,\\
\db{x_1,x_3}_a=&-\frac12x_3x_1\otimes\mathbf{1}_{p_0},\\
\db{x_2,x_3}_a=&-\frac12x_3x_2\otimes\mathbf{1}_{p_0}.
\end{align*}
\end{subfigure}
\begin{subfigure}{0.3\linewidth}
\centering
\begin {tikzpicture} [scale=1.5]
    \draw [dashed] (0,0) circle (1.0);

    \draw               (-0.98,-0.2) to [out=11, in=108.5] (0.32,-0.95);
    \draw [rotate=120]  (-0.98,-0.2) to [out=11, in=108.5] (0.32,-0.95);
    \draw [rotate=-120] (-0.98,-0.2) to [out=11, in=108.5] (0.32,-0.95);

    \draw [fill=white] ($0.37*(-0.5,0.866)$) circle (0.05);

    \draw [-latex] (-1,0) to [out=0, in=230] ($0.37*(-0.5,0.866) + 0.04*(-1,-1)$);
    \draw [-latex] ($0.32*(-0.5,0.866) + 0.03*(1,1)$) to [out=-10, in =240] (0.5,0.866);
    \draw [-latex] ($0.37*(-0.5,0.866) + 0.05*(0.5,-0.866)$) -- (0.5,-0.866);

    \draw (0.65,-1.016) node {$p_1$};
    \draw (0.65,1.016)  node {$p_2$};
    \draw (-1.15,0)     node {$p_3$};

    \draw (-0.3,0.483)  node {$p_0$};

    \draw (0.26,-0.18)  node {\small $x_1$};
    \draw (0.225,0.25) node {\small $x_2$};
    \draw (-0.3,0)     node {\small $x_3$};
\end {tikzpicture}
\begin{align*}
S_0^b=\{x_3^{-1},x_1,x_2\},
\end{align*}
\begin{align*}
\db{x_1,x_2}_b=&\frac12x_1\otimes x_2,\\
\db{x_1,x_3}_b=&\frac12x_3x_1\otimes\mathbf{1}_{p_0},\\
\db{x_2,x_3}_b=&\frac12x_3x_2\otimes\mathbf{1}_{p_0}.
\end{align*}
\end{subfigure}
\begin{subfigure}{0.3\linewidth}
\centering
\begin {tikzpicture} [scale=1.5]
    \draw [dashed] (0,0) circle (1.0);

    \draw               (-0.98,-0.2) to [out=11, in=108.5] (0.32,-0.95);
    \draw [rotate=120]  (-0.98,-0.2) to [out=11, in=108.5] (0.32,-0.95);
    \draw [rotate=-120] (-0.98,-0.2) to [out=11, in=108.5] (0.32,-0.95);

    \draw [fill=white] (0.37,0) circle (0.05);

    \draw [-latex] (-1,0) -- ($(0.37,0) - 0.04*(1,0)$);
    \draw [-latex] ($(0.37,0) + 0.05*(-0.5,0.866)$)  to [out=110, in=240] (0.5,0.866);
    \draw [-latex] ($(0.37,0) + 0.05*(-0.5,-0.866)$) to [out=-110, in=120] (0.5,-0.866);

    \draw (0.65,-1.016) node {$p_1$};
    \draw (0.65,1.016)  node {$p_2$};
    \draw (-1.15,0)     node {$p_3$};

    \draw (0.6,0) node {$p_0$};

    \draw (0.15,-0.3) node {\small $x_1$};
    \draw (0.15,0.3)  node {\small $x_2$};
    \draw (-0.3,0.1)  node {\small $x_3$};
\end {tikzpicture}
\begin{align*}
S_0^c=\{x_2,x_3^{-1},x_1\},
\end{align*}
\begin{align*}
\db{x_1,x_2}_c=&-\frac12x_1\otimes x_2,\\
\db{x_1,x_3}_c=&\frac12x_3x_1\otimes\mathbf{1}_{p_0},\\
\db{x_2,x_3}_c=&-\frac12x_3x_2\otimes\mathbf{1}_{p_0}.
\end{align*}
\end{subfigure}
\caption{Three choices of boundary component for white vertex}
\label{fig:ThreeChoices}
\end{figure}

As in the previous section, we fix a perfect planar network $\Gamma$ and the associated
ribbon surface $\Sigma_\Gamma$. Choose three scalars $w_{12}$, $w_{13}$, and $w_{23}$. At each white vertex,
let $x_1$, $x_2$, $x_3$ be the directed edges (in order) as depicted in Figure \ref{fig:ribbon_graph_pieces},
and define a double bracket by
\[ \db{x_1,x_2} = w_{12} (x_1 \otimes x_2), \quad \db{x_1,x_3} = w_{13} (x_3x_1 \otimes \mathbf{1}_V) \quad \db{x_2,x_3} = w_{23} (x_3x_2 \otimes \mathbf{1}_V) \]
Also choose three scalars $b_{12}$, $b_{13}$, and $b_{23}$, and at each black vertex, label the edges
from Figure \ref{fig:ribbon_graph_pieces} $y_1$, $y_2$, $y_3$, and define
\[ \db{y_1,y_2} = b_{12} (y_2 \otimes y_1), \quad \db{y_1,y_3} = b_{13} (\mathbf{1}_V \otimes y_1y_3) \quad \db{y_2,y_3} = b_{23} (\mathbf{1}_V \otimes y_2y_3) \]
We then compute a double bracket on general paths in $\Gamma$ using the double Leibniz identity (\ref{eq:DoubleLeibnizIdentityI})--(\ref{eq:DoubleLeibnizIdentityII}).
It follows from Proposition \ref{prop:GluingBivector} that the resulting double bracket on the subcategory of paths staring at sources and terminating at sinks must coincide with the surface double bracket on $\mathbf k\pi_1(\Sigma_\Gamma,p_1,\dots,p_m,q_1,\dots,q_n)$.

\begin {remark}
    The double bracket described in the previous section corresponds to the choices:
    \begin {align*}
        w_{12} &= \frac{1}{2} & w_{13} &= -\frac{1}{2} & w_{23} &= -\frac{1}{2} \\
        b_{12} &= \frac{1}{2} & b_{13} &= -\frac{1}{2} & b_{23} &= -\frac{1}{2}
    \end {align*}
\end {remark}

We will now describe a formula for the double bracket between paths, and then note that
some choices of the $w_{ij}$ and $b_{ij}$ give essentially the same bracket. We will use this
observation to choose a more convenient and simple set of $w_{ij}$ and $b_{ij}$. First we introduce some notation.

Define the quantities $B = b_{12}+b_{13}-b_{23}$ and $W = w_{12}+w_{13}-w_{23}$.
We will see that the double bracket can be expressed nicely in terms of $B$ and $W$.
If a path $f$ passes through an internal vertex $V$,
we write $f_V'$ for the part of $f$ ending at $V$, and $f_V''$ for the part after $V$, so that $f = f_V' f_V''$. Suppose two
paths $f$ and $g$ pass through a common vertex $V$, as in figure \ref{fig:paths_meet}.

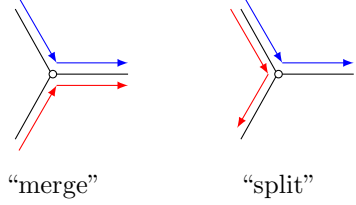
\begin {figure}
\centering
\caption {Merging and splitting of paths}
\label {fig:paths_meet}

\begin {tikzpicture}


    \draw [fill=white] (0,0) circle (0.05cm);
    \draw (0.05,0) -- (1,0);
    \draw (-0.5,0.866) -- (2.965-3,0.035);
    \draw (-0.5,-0.866) -- (2.965-3,-0.035);

    \draw (0,-1.5) node {``merge''};

    \draw [blue, -latex] (0.05,0.15) -- (1.0,0.15);
    \draw [blue, -latex] (-0.45,1.016) -- (0.05,0.15);

    \draw [red, -latex] (0.05,-0.15) -- (1.0,-0.15);
    \draw [red, -latex] (-0.45,-1.016) -- (0.05,-0.15);


    \begin {scope} [shift = {(3,0)}]
        \draw [fill=white] (0,0) circle (0.05cm);
        \draw (0.05,0)      -- (1,0);
        \draw (-0.5,0.866)  -- (2.965-3,0.035);
        \draw (-0.5,-0.866) -- (2.965-3,-0.035);

        \draw (0,-1.5) node {``split''};

        \draw [blue, -latex] (-0.45,1.016) -- (0.05,0.15);
        \draw [blue, -latex] (0.05,0.15)   -- (1.00,0.15);

        \draw [red, -latex] (-0.63661,0.866) -- (-0.13661,0);
        \draw [red, -latex] (-0.13661,0) -- (-0.55,-0.716);
    \end {scope}

\end {tikzpicture}
\end {figure}

We define the
quantity $\alpha_V(f,g)$, depending on the color of $V$ and how $f$ and $g$ meet, given in the table below.
In the table, $f$ is the blue path (on the left) and $g$ is red (on the right) from figure \ref{fig:paths_meet}.

\begin {center}
\begin {tabular}{|c|c|c|} \hline
    color     & type  & $\alpha_V(f,g)$ \\ \hline
    $\circ$   & merge & $W$ \\
    $\bullet$ & merge & $B$ \\
    $\circ$   & split & $-W$ \\
    $\bullet$ & split & $-B$ \\ \hline
\end {tabular}
\end {center}

\begin {lemma} \cite{Ovenhouse'2020}
    Let $f$ and $g$ be two paths, which begin and end on the boundary of the disk. Their double bracket is given by
    \[ \db{f,g} = \sum_V \alpha_V(f,g) \, (g_V' f_V'' \otimes f_V' g_V'') \]
    The sum is over all vertices $V$ where $f$ and $g$ either merge or split.
\end {lemma}

The significance of this lemma is the following. We will primarily be concerned with paths which begin at sources on the boundary
and terminate at sinks on the boundary. In this case, the double bracket between two paths depends only on the quantities $B$ and $W$, and
not on the individual values of $w_{ij}$ and $b_{ij}$. We can therefore choose different values of these constants which give the same $B$
and $W$, and the double bracket will be the same on paths between boundary vertices. As was mentioned earlier, the double bracket on $\Sigma_\Gamma$
described in the previous section has the values $B = W = \frac{1}{2}$. According to the lemma, we may locally define the bracket
by setting $w_{12} = b_{12} = \frac{1}{2}$, and the rest of the $w_{ij}$ and $b_{ij}$ to zero, without changing the double brackets
between paths connecting boundary vertices.

\subsection {The Twisted Ribbon Graph}

Given a perfect planar network $\Gamma$, we constructed the ribbon graph $\Sigma_\Gamma$ in the previous section,
and saw that the particular double bracket on $\mathcal{C}$ with $B=W=\frac{1}{2}$ gave the standard Goldman bracket on $\Sigma_\Gamma$.
Now we construct another surface, closely related to $\Sigma_\Gamma$. It is called the \emph{twisted ribbon graph}, and
is denoted $\widehat{\Sigma}_\Gamma$. It is built in much the same way as $\Sigma_\Gamma$, except that for any edge connecting
two vertices of different colors, we give the ribbon a half-twist before glueing. In this way, we think of the black vertices
as being ``upside-down''.

It is easy to see that the double bracket for $\widehat{\Sigma}_\Gamma$, as described in Section 2, is given by the
member of our two-parameter family with $W = \frac{1}{2}$ and $B = - \frac{1}{2}$. As before, to make calculations simpler,
we will assume that $w_{12} = \frac{1}{2}$ and $b_{12} = -\frac{1}{2}$, with all other local brackets being zero.

From now on, unless stated otherwise, given a perfect planar network $\Gamma$, our standard choice of double bracket
will be this one, with $W = \frac{1}{2}$ and $B = -\frac{1}{2}$, corresponding to the surface $\widehat{\Sigma}_\Gamma$.

\subsection {Cylindrical Networks}

We now consider a version of perfect networks embedded on a cylinder (or annulus) instead of a disk.
These were studied in \cite{GSV'2012}.

\begin {definition}
    A directed graph $\Gamma$ embedded on a cylinder is called a \emph{perfect cylindrical network} if
    \begin {itemize}
        \item boundary vertices are univalent
        \item internal vertices are trivalent
        \item internal vertices are neither sources nor sinks
    \end {itemize}
\end {definition}

We make a simplifying assumption as in the planar case: we require that all sources are on one boundary component, and all sinks are on the other boundary component.
As before, we may form the ribbon surfaces $\Sigma_\Gamma$ and $\widehat{\Sigma}_\Gamma$, and consider the corresponding double brackets.
Again we take the bracket corresponding to $\widehat{\Sigma}_\Gamma$ as our standard choice.

In order to define boundary measurements, we introduce a new ingredient. Choose some smooth oriented curve whose endpoints are on different boundary components.
We call this curve the \emph{cut}. We require that the cut avoids the vertices of $\Gamma$, and intersects edges transversally. After cutting the cylinder
along this curve, it becomes a rectangle (or a disk). The condition that all sources are on one boundary component, and all sinks on the other,
guarantees that after making this cut, all sources are on the left edge of the resulting rectangle, and all sinks on the right edge. We then number the sources and sinks
in increasing order from top to bottom, as in the planar case.

\begin {definition}
    Let $\Gamma$ be a perfect cylindrical network with $m$ sources and $n$ sinks. We define the $m$-by-$n$ boundary measurement matrix $B_\Gamma(\lambda)$,
    whose entries are Laurent polynomials in the parameter $\lambda$, with coefficients in $\mathcal{C}$, as follows. The entries are
    \[ b_{ij}(\lambda) = \sum_{p \colon i \to j} \lambda^d \mathrm{wt}(p) \]
    The weight $\mathrm{wt}(p)$ is as before, and the exponent $d$ is the oriented intersection index of $p$ with the cut.
\end {definition}

The analogue of Propsition \ref{prop:bm_matrix_concatenation} is true in this case too. If $\Gamma_1$ has $m$ sources and $k$ sinks, and $\Gamma_2$
has $k$ sources and $n$ sinks, then we may glue the sink end of $\Gamma_1$ to the source end of $\Gamma_2$ to get another network $\Gamma$, and
\[ B_\Gamma(\lambda) = B_{\Gamma_1}(\lambda) B_{\Gamma_2}(\lambda) \]

\section{$r$-Matrix Bracket}

In this section, we will show that the double brackets between elements of the boundary measurement matrix $B_\Gamma$
can be described in terms of $r$-matrices. This generalizes results from \cite{GSV'2009} and \cite{GSV'2012},
where analogous $r$-matrix formulas were given in the commutative case, where the edges of the graph were weighted by scalars.

Hereinafter we use boxed times $\boxtimes$ for the tensor product of matrices (over $\mathcal C$, $\mathcal C\otimes\mathcal C$, or $\mathbf k$ depending on which matrices we apply it to). We reserve notation $\otimes=\otimes_{\mathbf k}$ for the tensor product of morphisms in $\mathcal C$. In addition, since $\mathcal C$ is $\mathbf k$-linear we can multiply elements of $\mathrm{Mat}_n(\mathcal C)$ and $\mathrm{Mat}_n(\mathcal C\otimes\mathcal C)$ by scalar matrices of an appropriate size.
\begin {definition}
    Let $\mathcal C$ be a small $\mathbf k$-linear category (a.k.a. algebra with many objects) equipped with a double bracket $\db{-,-}$. Define an operation
    \begin{align*}
    \mathrm{Mat}_n(\mathcal C) \times \mathrm{Mat}_n(\mathcal C) \to \mathrm{Mat}_{n}(\mathcal C \otimes \mathcal C)\boxtimes\mathrm{Mat}_n(\mathcal C\otimes\mathcal C),
    \end{align*} also denoted by $\db{-,-}$,
    as follows. For matrices $X,Y \in \mathrm{Mat}_n(\mathcal C)$:
    \[ \db{X,Y} = \sum_{ijk\ell} \db{x_{ij},y_{k\ell}} E_{ij} \boxtimes E_{k\ell}. \]
    Here, $E_{ij}$ are the elementary matrices with a $1$ in the $i,j$ place and zeros elsewhere.
\end {definition}

\begin {definition}
    For a matrix $M \in \mathrm{Mat}_n(\mathcal C)$, define the ``left and right'' versions of $M$ over $\mathcal C \otimes \mathcal C$ as follows.
    Let $M_L \in \mathrm{Mat}_n(\mathcal C \otimes \mathcal C)$ represent the matrix
    with $i,j$-entry $m_{ij} \otimes 1$.
    Similarly define $M_R$ to have entries $1 \otimes m_{ij}$.
\end {definition}

\begin {proposition} \label{prop:matrix_leibniz_identities}
    The $\db{-,-}$ operation on matrices satisfies the following Leibniz identities:
    \begin {enumerate}
        \item $\displaystyle \db{X,YZ} = \db{X,Y}(\mathbf 1 \boxtimes Z_R) + (\mathbf 1 \boxtimes Y_L) \db{X,Z},$\qquad\qquad whenever $Y,Z$ are composable,
        \item $\displaystyle \db{XY,Z} = \db{X,Z}(Y_L \boxtimes\mathbf 1) + (X_R \boxtimes\mathbf 1) \db{Y,Z},$\qquad\qquad whenever $X,Y$ are composable.
    \end {enumerate}
    Here $\mathbf 1\in\mathrm{Mat}_n(\mathbf k)$ stands for an identity matrix.
\end {proposition}

\begin {corollary}
    For matrices $X$ and $Y$ over $\mathcal C$,
    \begin{equation}
        \db{X^k,Y^\ell} = \sum_{i=0}^{k-1} \sum_{j=0}^{\ell-1} (X_R^i \boxtimes Y_L^j) \db{X,Y} (X_L^{k-1-i} \boxtimes Y_R^{\ell-1-j}).
    \label{eq:MatrixDoubleLeibniz}
    \end{equation}
    \label{cor:MatrixDoubleLeibnizPowers}
\end {corollary}


\subsection {$r$-Matrix Bracket for Planar Networks}

In this section, we show that double brackets for planar networks can be described in terms of an $r$-matrix formula using the boundary measurement matrix.
This was shown, in the commutative case, for square matrices (same number of sources and sinks), in \cite{GSV'2009}. (For quantum case see \cite{ChekhovShapiro'2020})

For each $n$, we define the element $r_n \in \mathrm{Mat}_n\mathbb Q \otimes \mathrm{Mat}_n\mathbb Q$ as
\begin{align}
 r_n = \frac{1}{2} \sum_{i < j} E_{ji} \boxtimes E_{ij} - E_{ij} \boxtimes E_{ji}
\label{eq:RMatrixDisc}
\end{align}

Our main theorem for planar networks is the following.

\begin {theorem} \label{thm:planar_r_matrix_formula}
    Suppose $\Gamma$ is an acyclic perfect planar network with $m$ sources and $n$ sinks separated from each other on the boundary. Denote the corresponding boundary measurement matrix $B = B_\Gamma$, then
    \begin {equation} \label{eq:twisted_commutator_formula}
        \db{B,B} \, = r_m (B_L \boxtimes B_R) - \Big( (B_L \boxtimes B_R) r_n \Big)^\tau
    \end {equation}
    In particular, if $B$ is square (with $m=n$), then
    \[ \db{B,B} \, = [r_n, B_L \boxtimes B_R]_\tau \]
    where $[A,B]_\tau=AB-(BA)^\tau$ stands for the twisted commutator.
\end {theorem}

We will prove Theorem \ref{thm:planar_r_matrix_formula} in steps. First we will show that it holds for the elementary pieces consisting of
a single trivalent vertex, and then we will show that the $r$-matrix bracket is preserved under the fusion process.

First we consider the elementary networks $\Gamma_\circ$ and $\Gamma_\bullet$, which consist of a single white (resp. black) trivalent vertex.

\begin {lemma}
    The result of Theorem \ref{thm:planar_r_matrix_formula} holds for $\Gamma_\circ$ and $\Gamma_\bullet$.
\end {lemma}
\begin {proof}
    Consider first $\Gamma_\circ$. It has one source and two sinks, and so $B_{\Gamma_\circ}$ is the $1$-by-$2$ matrix
    \[ B = \begin{pmatrix} x_3x_2 & x_3x_1 \end{pmatrix} \]
    The brackets are given by
    \begin {align*}
        \db{B,B} &= \frac{1}{2} \begin{pmatrix} 0 & -x_3x_2 \otimes x_3x_1 & x_3x_1 \otimes x_3x_2 & 0 \end{pmatrix} \\
                 &= \frac{1}{2} \begin{pmatrix} 0 & -b_{11} \otimes b_{12} & b_{12} \otimes b_{11} & 0 \end{pmatrix}
    \end {align*}
    In this case, $r_1 = (0)$ is the $1$-by-$1$ zero matrix, and
    \begin {align*}
        (B_L \boxtimes B_R) r_2 &= \frac{1}{2} \begin{pmatrix} b_{11} \otimes b_{11} & b_{11} \otimes b_{12} & b_{12} \otimes b_{11} & b_{12} \otimes b_{12} \end{pmatrix}
                                 \begin{pmatrix} 0 & 0 & 0 & 0 \\ 0 & 0 & -1 & 0 \\ 0 & 1 & 0 & 0 \\ 0 & 0 & 0 & 0 \end{pmatrix} \\
                              &= \frac{1}{2} \begin{pmatrix} 0 & b_{12} \otimes b_{11} & -b_{11} \otimes b_{12} & 0 \end{pmatrix}
    \end {align*}
    We see that $\db{B,B} = -((B_L \boxtimes B_R) r_2)^\tau$, and so the formula holds for $\Gamma_\circ$.

    Now we consider $\Gamma_\bullet$. There are two sources and one sink, so the boundary measurement matrix is the $2$-by-$1$ matrix
    \[ B = \begin{pmatrix} y_1y_3 \\[1.2ex] y_2y_3 \end{pmatrix} \]
    The bracket relations are
    \[ \db{B,B} \, = \frac{1}{2} \begin{pmatrix} 0 \\[1.2ex] -y_2y_3 \otimes y_1y_3 \\[1.2ex] y_1y_3 \otimes y_2y_3 \\[1.2ex] 0 \end{pmatrix}
                   = \frac{1}{2} \begin{pmatrix} 0 \\[1.2ex] -b_{21} \otimes b_{11} \\[1.2ex] b_{11} \otimes b_{21} \\[1.2ex] 0 \end{pmatrix}
    \]
    On the other hand,
    \[
        r_2 (B_L \boxtimes B_R) \, = \frac{1}{2} \begin{pmatrix} 0 & 0 & 0 & 0 \\ 0 & 0 & -1 & 0 \\ 0 & 1 & 0 & 0 \\ 0 & 0 & 0 & 0 \end{pmatrix}
                                 \begin{pmatrix} b_{11} \otimes b_{11} \\[1.2ex] b_{11} \otimes b_{21} \\[1.2ex] b_{21} \otimes b_{11} \\[1.2ex] b_{21} \otimes b_{21} \end{pmatrix}
        = \frac{1}{2} \begin{pmatrix} 0 \\[1.2ex] -b_{21} \otimes b_{11} \\[1.2ex] b_{11} \otimes b_{21} \\[1.2ex] 0 \end{pmatrix}
    \]
    This confirms that the formula holds for $\Gamma_\bullet$.
\end {proof}

\begin{lemma}\label{lem:gluing}
    Any trivalent network $\Gamma$ in a disk with non-interlacing univalent sources and sinks on the boundary
    and with no sources or sinks inside the disk and no oriented cycles can be built by consecutively gluing trivalent pieces
    $\Gamma_\circ$ and $\Gamma_\bullet$ to the boundary.
\end{lemma}

\begin{proof}
We prove the Lemma by induction on the number of inner (trivalent) vertices of the network, starting with a network with one inner vertex.
If $\Gamma$ has only one inner vertex, then it has to be of type $\Gamma_\circ$ or $\Gamma_\bullet$ and the statement is proved.

Assume that $\Gamma$ has more than one inner vertex. Consider the set of sinks $I$ on the boundary.
Let $\widehat V$ be the subset of inner vertices which are connected to $I$ by one arc.
For $\hat v\in \widehat V$ we consider three possible cases.

\begin{enumerate}
\item[1] $\hat v$ is white, there is one incoming arrow $\alpha(\hat v)$ to $\hat v$ and two outgoing arrows
         $\beta(\hat v)$ and $\gamma(\hat v)$, both connecting $\hat v$ to two neighboring sinks $B$ and $C$
         (see Figure~\ref{fig:ThreeInnerChoices} Type 1). Cutting off the piece of the disk containing $B,C$, and $\hat v$
         along the dashed curve, we obtain a network with one less boundary sink and one less inner vertex,
         satisfying all conditions of the lemma and, thus, this case is proved by induction.
\item[2] $\hat v$ is a black vertex with two incoming arrows $\alpha(\hat v)$ and $\beta(\hat v)$ and one outgoing arrow $\gamma(\hat v)$
         connecting $\hat v$ to sink $C$ (see Figure~\ref{fig:ThreeInnerChoices} Type 2)). Again,
         cutting off the sink $C$ and vertex $\hat v$ along the dashed curve, we obtain a disk with one more sink and one less inner vertex and we are done by induction.
\item[3] $\hat v$ is a white vertex with one incoming arrow $\alpha(\hat v)$, one outgoing arrow $\beta(\hat v)$ connecting $\hat v$ with sink $B$,
         and another outgoing arrow $\gamma(\hat v)$ not connecting $\hat v$ to another sink (see Figure~\ref{fig:ThreeInnerChoices} Type 3).
         Note that cutting along the dotted curve might break the condition that sources and sinks are separated.
         Hence, such a simple inductive argument does not work in this case.
\end{enumerate}

We will prove now that the set $\widehat V$ contains at least one vertex of type 1 or 2.
Indeed, assume that $\widehat V$ contains only vertices of type 3. We will show that then
$\Gamma$ must have an oriented cycle. Indeed, remove from $\Gamma$ each arc $\beta(\hat v)$ for all $\hat v\in \widehat V$.
The result is an oriented graph in the disk with no sinks, and only sources on the boundary.
Such graph must have an oriented cycle because any oriented path can be extended (any vertex has an outgoing arrow).
Then, for any $\Gamma$ satisfying the lemma's condition, we can always apply induction step of type 1 or 2.

\begin{figure}
\begin{subfigure}{0.3\linewidth}
\centering
\begin {tikzpicture} [scale=1.5]
    \draw [ultra thick] (0,0) circle (1.0);


    \draw [fill=white] ($0.37*(-0.5,-0.866)$) circle (0.05);

    \draw [-latex] (-0.6,0) to [out=0, in=130] ($0.37*(-0.5,-0.866) + 0.04*(-1,1)$);
    \draw [-latex] ($0.32*(-0.5,-0.866) + 0.05*(0,-1)$) -- (0,-1);
    \draw [-latex] ($0.37*(-0.5,-0.866) + 0.05*(0,-1)$) to [out=20, in=120] (0.5,-0.866);

    \draw (0.65,-1.016) node {$B$};
    \draw (0,-1.126)  node {$C$};

    \draw (-0.3,-0.4) node {$\hat v$};

    \draw (0.2,-0.35) node {\small $\beta$};
    \draw (-0.,-0.7) node {\small $\gamma$};
    \draw (-0.35,-0.05)    node {\small $\alpha$};
 \draw       [dashed]   (-0.08,-1) to [out=120, in=180] (-0.35,0.05) to  [out=0, in=60]  (0.65,-0.8);

\end {tikzpicture}
\begin{align*}
\text{Type 1.}
\end{align*}

\end{subfigure}
\begin{subfigure}{0.3\linewidth}
\centering
\begin {tikzpicture} [scale=1.5]
    \draw [ultra thick] (0,0) circle (1.0);


    \draw [fill=black] ($0.37*(-0.5,-0.866)$) circle (0.05);

    \draw [-latex] (-0.6,0) to [out=0, in=130] ($0.37*(-0.5,-0.866) + 0.04*(-1,1)$);
    \draw [-latex] ($0.32*(-0.5,-0.866) + 0.05*(0,-1)$) -- (0,-1);
    \draw [-latex] ((0.,0.2) to  [out=240, in=90]  ($0.37*(-0.5,-0.866) + 0.05*(0,1)$);

    \draw (0,-1.126)  node {$C$};

    \draw (-0.3,-0.4) node {$\hat v$};

    \draw (0.04,-0.05) node {\small $\beta$};
    \draw (-0.,-0.7) node {\small $\gamma$};
    \draw (-0.35,-0.05)    node {\small $\alpha$};
 \draw       [dashed]   (-0.08,-1) to [out=120, in=180] (-0.35,0.05) to  [out=0, in=60]  (0.08,-1);

\end {tikzpicture}
\begin{align*}
\text{Type 2.}
\end{align*}

\end{subfigure}
\begin{subfigure}{0.3\linewidth}
\centering
\begin {tikzpicture} [scale=1.5]
    \draw [ultra thick] (0,0) circle (1.0);


    \draw [fill=white] ($0.37*(-0.5,-0.866)$) circle (0.05);

    \draw [-latex] (-0.6,0) to [out=0, in=130] ($0.37*(-0.5,-0.866) + 0.04*(-1,1)$);
    \draw [-latex] ($0.32*(-0.5,-0.866) + 0.05*(0,-1)$) -- (0,-1);
    \draw [-latex] ($0.37*(-0.5,-0.866) + 0.05*(0,1)$) to [out=80, in=240] (0.,0.2);

    \draw (0,-1.126)  node {$B$};

    \draw (-0.3,-0.4) node {$\hat v$};

    \draw (0.04,-0.05) node {\small $\gamma$};
    \draw (-0.,-0.7) node {\small $\beta$};
    \draw (-0.35,-0.05)    node {\small $\alpha$};
 \draw       [dotted]   (-0.08,-1) to [out=120, in=180] (-0.35,0.05) to  [out=0, in=60]  (0.08,-1);

\end {tikzpicture}
\begin{align*}{l}
\text{Type 3.} \\
\end{align*}

\end{subfigure}
\caption{Three types of an inner vertex connected to a sink.  Note that the figure  of type 3 shows only one possible cyclic order of $\alpha,\beta,\gamma$. One can also exchange order of $\alpha$ and $\gamma$.
}
\label{fig:ThreeInnerChoices}
\end{figure}
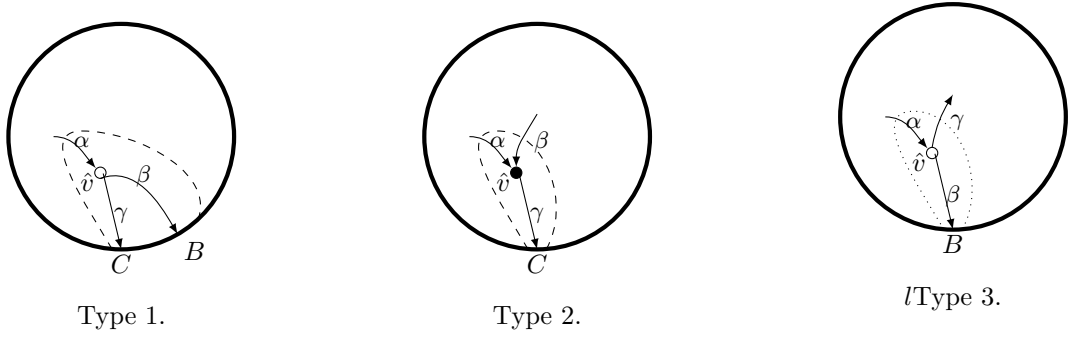

\end{proof}

Because any network $\Gamma$ that we consider can be built from these trivalent pieces $\Gamma_\circ$ and $\Gamma_\bullet$ by gluing, Theorem \ref{thm:planar_r_matrix_formula}
follows by induction from the next lemma.

\begin {lemma} \label{lem:bracket_induction}
    Suppose that $\Gamma_1$ and $\Gamma_2$ are perfect planar networks with boundary measurement matrices $X = B_{\Gamma_1}$ and $Y = B_{\Gamma_2}$,
    of sizes $m$-by-$k$ and $k$-by-$n$ respectively.
    Let $\Gamma$ be the network obtained by glueing the sinks of $\Gamma_1$ to the sources of $\Gamma_2$, with boundary measurement matrix $B = XY$.
    If the statement of Theorem \ref{thm:planar_r_matrix_formula} holds for both $\Gamma_1$ and $\Gamma_2$, then it holds for $\Gamma$ as well.
\end {lemma}
\begin {proof}
    First note that because boundary vertices are univalent, it follows from the local definition of the bracket that $\db{X,Y} = 0$.
    Then since $B = XY$, we have by Proposition \ref{prop:matrix_leibniz_identities} combined with Corollary \ref{cor:GluingCommute} that
    \[ \db{B,B} \, = \db{XY, \, XY} = \db{X,X} (Y_L \boxtimes\mathbf 1)(\mathbf1 \boxtimes Y_R) + (\mathbf 1 \boxtimes X_L)(X_R \boxtimes\mathbf 1) \db{Y,Y} \]
    Using the assumption that the $r$-matrix formula holds for $X$ and $Y$, we get
    \begin{equation}
    \begin{aligned}
        \db{B,B} &= r_m (X_L \boxtimes X_R)(Y_L \boxtimes\mathbf 1)(\mathbf 1 \boxtimes Y_R) - (X_R \boxtimes X_L) r_k (Y_L \boxtimes\mathbf 1)(\mathbf 1 \boxtimes Y_R) \\
        &\phantom{=} + (\mathbf 1 \boxtimes X_L)(X_R \boxtimes\mathbf 1) r_k (Y_L \boxtimes Y_R) - (\mathbf 1 \boxtimes X_L)(X_R \boxtimes\mathbf 1)(Y_R \boxtimes Y_L) r_n
    \end{aligned}
    \label{eq:BracketBB}
    \end{equation}
    Although for general matrices $A,B,C,D$ with noncommuting entries $(A \boxtimes B)(C \boxtimes D) \neq AC \boxtimes BD$,
    it is true however, that $(Y_L \boxtimes\mathbf 1)(\mathbf 1 \boxtimes Y_R) = (Y_L \boxtimes Y_R)$, and similarly that $(\mathbf 1 \boxtimes X_R)(X_L \boxtimes\mathbf 1) = X_L \boxtimes X_R$.
    So the middle two terms in the equation above cancel. Also, it is true in general that $(A_L \boxtimes B_R)(C_L \boxtimes D_R) = A_LC_L \boxtimes B_RD_R$.
    So (\ref{eq:BracketBB}) simplifies to
    \begin {align*}
        \db{B,B} &= r_m (X_LY_L \boxtimes X_RY_R) - (X_RY_R \boxtimes X_LY_L) r_n \\
                 &= r_m (B_L \boxtimes B_R) - \Big( (B_L \boxtimes B_R) r_n \Big)^\tau
    \end {align*}
\end {proof}

\subsection{$r$-Matrix Bracket for Cylindrical Networks}

In this subsection we prove that the double brackets for cylindrical networks can be presented in an analogous way to Theorem \ref{thm:planar_r_matrix_formula}.
In this case, we use (for each $n$) the trigonometric $r$-matrix for $\mathrm{GL}(n)$:
\begin{equation}
r_n(\lambda,\mu) = \frac{1}{2} \, \frac{\mu+\lambda}{\mu-\lambda} \sum_{m=1}^n E_{mm} \boxtimes E_{mm} + \frac{1}{\mu-\lambda} \sum_{1\leqslant i < j\leqslant n}\left(\lambda E_{ij} \boxtimes E_{ji} + \mu E_{ji}\boxtimes E_{ij}\right).
\label{eq:TrigonometricRMatrix}
\end{equation}
Here the entries of $r_n(\lambda,\mu)$ are elements of $\mathbf k = \mathbb C(\lambda,\mu)$, the ground field, which for this subsection is chosen to
be the ring of rational functions in two variables $\lambda$ and $\mu$ commonly referred to as \textit{spectral parameters}.

The analogue of Theorem \ref{thm:planar_r_matrix_formula} for cylindrical networks is the following.

\begin {theorem} \label{thm:cylindrical_r_matrix_formula}
    Let $\Gamma$ be a perfect cylindrical network with $m$ sources and $n$ sinks. Then
    \begin{subequations}
    \begin{equation}
    \db{B(\lambda),B(\mu)} \, = r_m(\lambda,\mu) (B_L(\lambda) \boxtimes B_R(\mu)) - \Big( (B_L(\lambda) \boxtimes B_R(\mu) \, r_n(\lambda,\mu) \Big)^\tau
    \label{eq:TwistedCommutatorSpectralParameterGeneral}
    \end{equation}
    In particular, if $B(\lambda)$ is a square matrix, then this can be written as a twisted commutator:
    \begin{equation}
    \db{B(\lambda),B(\mu)} \, = \left[ r(\lambda,\mu), \, B_L(\lambda) \boxtimes B_R(\mu) \right]_\tau
    \label{eq:TwistedCommutatorSpectralParameterSquare}
    \end{equation}
    \end{subequations}
\end {theorem}

Again we prove this in two steps by first proving the result for elementary building blocks, and then showing that the twisted commutator
expression above is preserved under gluing. Recall that $\Gamma_\circ$ and $\Gamma_\bullet$ are the elementary networks consisting of a single
trivalent vertex, as depicted in Figure \ref{fig:ribbon_graph_pieces}.

\begin {lemma}
    The statement of Theorem \ref{thm:cylindrical_r_matrix_formula} holds for the elementary networks $\Gamma_\circ$ and $\Gamma_\bullet$.
\end {lemma}
\begin {proof}
    The proof will be almost the same as in the planar case, but we will also need to consider the cases that one of the three edges crosses the cut.
    Note that Lemma ~\ref{lem:gluing}  holds by the same arguments for  any oriented trivalent network in a cylinder without oriented cycles and with no sources or sinks inside.

    First let us consider $\Gamma_\circ$, and assume that the network does not cross the cut. Then as in the planar case, the boundary measurement matrix is:
    \[ B(\lambda) = \begin{pmatrix} x_3x_2 & x_3x_1 \end{pmatrix}, \]
    and the brackets are given by
    \begin {align*}
        \db{B(\lambda),B(\mu)} &= \frac{1}{2} \begin{pmatrix} 0 & -x_3x_2 \otimes x_3x_1 & x_3x_1 \otimes x_3x_2 & 0 \end{pmatrix} \\
                 &= \frac{1}{2} \begin{pmatrix} 0 & -b_{11} \otimes b_{12} & b_{12} \otimes b_{11} & 0 \end{pmatrix}
    \end {align*}
    In the cylindrical case, $r_1(\lambda,\mu)$ is the $1$-by-$1$ matrix with coefficient $\frac{\mu+\lambda}{2(\mu-\lambda)}$. So we have
    \[ r_1(\lambda,\mu) (B_L(\lambda) \boxtimes B_R(\mu)) = \frac{\mu+\lambda}{2(\mu-\lambda)}
       \begin {pmatrix} b_{11} \otimes b_{11} & b_{11} \otimes b_{12} & b_{12} \otimes b_{11} & b_{12} \otimes b_{12} \end {pmatrix}
    \]
    On the other hand, we have
    \begin {align*}
        (B_L(\lambda) \boxtimes B_R(\mu)) r_2(\lambda,\mu) &= \frac{1}{2(\mu-\lambda)} \begin {pmatrix} b_{11} \otimes b_{11} & b_{11} \otimes b_{12} & b_{12} \otimes b_{11} & b_{12} \otimes b_{12} \end {pmatrix}
        \begin {pmatrix} \mu+\lambda & 0 & 0 & 0 \\ 0 & 0 & 2\lambda & 0 \\ 0 & 2\mu & 0 & 0 \\ 0 & 0 & 0 & \mu+\lambda \end {pmatrix} \\
        &= \frac{1}{2(\mu-\lambda)} \begin {pmatrix} (\mu+\lambda)(b_{11} \otimes b_{11}) & 2\mu (b_{12} \otimes b_{11}) & 2\lambda (b_{11} \otimes b_{12}) & (\mu+\lambda)(b_{12} \otimes b_{12}) \end {pmatrix}
    \end {align*}
    After applying $\tau$ and subtracting from the equation above, we see that the formula holds for $\Gamma_\circ$. Note that although $\lambda$ and $\mu$ appear in each term,
    after subtracting all nonzero coefficients become $\frac{\lambda-\mu}{\mu-\lambda}$ or $\frac{\mu-\lambda}{\mu-\lambda}$, and it agrees with the planar case.

    The calculation for $\Gamma_\bullet$ is similar, and we leave it to the reader.

    Now we must consider the cases where the elementary pieces cross the cut. In this case, the boundary measurement matrices will explicitly depend on the parameter $\lambda$.
    Consider first $\Gamma_\bullet$, and suppose that the edge labelled $y_2$ crosses the cut. If the cut is oriented from left-to-right, then $y_2$ gets a factor of $\lambda$.
    The boundary measurement matrix in this case is
    \[ B(\lambda) = \begin{pmatrix} \lambda \, y_2y_3 \\[1.5ex] y_1y_3 \end {pmatrix} \]
    Note that the rows are permuted compared to the planar version of $\Gamma_\bullet$ because one of the edges crosses the cut.
    The bracket relations give us
    \[ \db{B(\lambda), B(\mu)} \, = \frac{1}{2} \begin{pmatrix} 0 \\[1.5ex] \lambda \, y_1y_3 \otimes y_2y_3 \\[1.5ex] -\mu \, y_2y_3 \otimes y_1y_3 \\[1.5ex] 0 \end{pmatrix}
       = \frac{1}{2} \begin{pmatrix} 0 \\[1.5ex] b_{21}(\mu) \otimes b_{11}(\lambda) \\[1.5ex] -b_{11}(\mu) \otimes b_{21}(\lambda) \\[1.5ex] 0 \end{pmatrix}
    \]
    The claim is that this is equal to
    \[
        r_2(\lambda,\mu) (B_L(\lambda) \boxtimes B_R(\mu)) - \Big( (B_L(\lambda) \boxtimes B_R(\mu)) r_1(\lambda,\mu) \Big)^\tau
        = \frac{1}{2(\mu-\lambda)} \begin{pmatrix} (\mu+\lambda) \left( b_{11}(\lambda) \otimes b_{11}(\mu) - b_{11}(\mu) \otimes b_{11}(\lambda) \right) \\[1.5ex]
                                                   2\lambda (b_{21}(\lambda) \otimes b_{11}(\mu)) - (\mu+\lambda)(b_{21}(\mu) \otimes b_{11}(\lambda)) \\[1.5ex]
                                                   2\mu     (b_{11}(\lambda) \otimes b_{21}(\mu)) - (\mu+\lambda)(b_{11}(\mu) \otimes b_{21}(\lambda)) \\[1.5ex]
                                                   (\mu+\lambda) \left( b_{21}(\lambda) \otimes b_{21}(\mu) - b_{21}(\mu) \otimes b_{21}(\lambda) \right)
                                    \end {pmatrix}
    \]
    There are four equations to check to verify this. The first condition is that $b_{11}(\lambda) \otimes b_{11}(\mu) = b_{11}(\mu) \otimes b_{11}(\lambda)$.
    This is true, since both are equal to $\lambda \mu \, y_2y_3 \otimes y_2y_3$. Likewise, the fourth condition, that
    $b_{21}(\lambda) \otimes b_{21}(\mu) = b_{21}(\mu) \otimes b_{21}(\lambda)$, is also true, since both are equal to $y_1y_3 \otimes y_1y_3$.
    The second condition is that
    \[ \frac{\lambda}{\mu-\lambda} (b_{21}(\lambda) \otimes b_{11}(\mu)) - \frac{\mu+\lambda}{2(\mu-\lambda)}(b_{21}(\mu) \otimes b_{11}(\lambda)) = \frac{1}{2} (b_{21}(\mu) \otimes b_{11}(\lambda)) \]
    The left-hand side of this equation reduces to
    \[ \frac{1}{2} \, \frac{1}{\mu-\lambda} \left( 2 \lambda \mu - (\mu+\lambda) \lambda \right) (y_1y_3 \otimes y_2y_3)
       = \frac{\lambda}{2} (y_1y_3 \otimes y_2y_3)
       = \frac{1}{2} (b_{21}(\mu) \otimes b_{11}(\lambda))
    \]
    The cases for the other ways in which the edges can cross the cut, and the remaining cases for $\Gamma_\circ$, are all similar.
\end {proof}

Now that we have verified the formula for elementary pieces, we will prove a lemma analogous to Lemme \ref{lem:bracket_induction} for cylindrical networks.
This will prove Theorem \ref{thm:cylindrical_r_matrix_formula}.

\begin {lemma}
    Suppose that $\Gamma_1$ and $\Gamma_2$ are perfect cylindrical networks with boundary measurement matrices $X(\lambda) = B_{\Gamma_1}(\lambda)$ and $Y(\lambda) = B_{\Gamma_2}(\lambda)$,
    of sizes $m$-by-$k$ and $k$-by-$n$ respectively.
    Let $\Gamma$ be the network obtained by glueing the sinks of $\Gamma_1$ to the sources of $\Gamma_2$, with boundary measurement matrix $B(\lambda) = X(\lambda)Y(\lambda)$.
    If the statement of Theorem \ref{thm:cylindrical_r_matrix_formula} holds for both $\Gamma_1$ and $\Gamma_2$, then it holds for $\Gamma$ as well.
\end {lemma}
\begin {proof}
    The proof is identical to Lemma \ref{lem:bracket_induction}.
\end {proof}

\section{Lax Equation and Trace Hamiltonians}
\label{sec:LaxEquationTraceHomiltonians}

For this section, let $\Gamma$ be a fixed perfect cylindrical network with equal number of inputs and outputs, say $N$.
Fix ground field $\mathbf k=\mathbb C(\lambda,\mu)$ consisting of rational functions in two variables $\lambda,\mu$ and denote
by $\mathcal C^{(1)}$ the $\mathbf k$-linear category generated by paths in $\Gamma$ starting/terminating at one of the boundary points.
Let $B(\lambda) \in \mathrm{Mat}_n \mathcal C^{(1)}$ be the corresponding boundary measurement matrix.
From Section \ref{sec:PlanarAndCylindricalNetworks} we know that $\mathcal C^{(1)}$ can be equipped with a double bracket which we denote by $\db{,}$.
We will show in Section \ref{sec:DoubleQuasiJacobi} that the bracket is quasi Poisson.

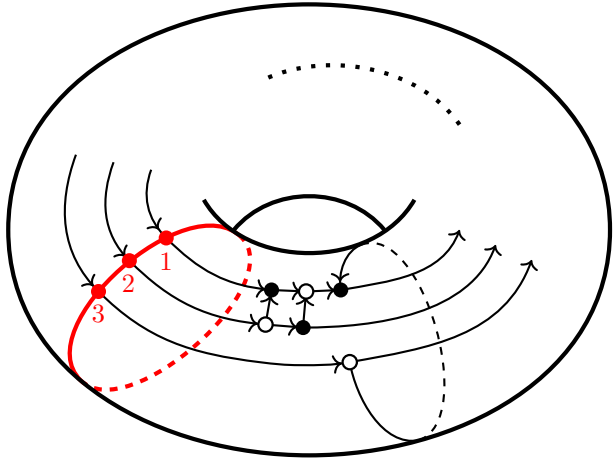
\begin{figure}
\begin{tikzpicture}[scale=2.0]
  \draw[ultra thick,red] (-0.45,-0.05) to[out=135,in=135,looseness=0.9] (-1.52,-1);
  \draw[ultra thick,red,dashed] (-1.52,-1) to[out=-45,in=-45,looseness=0.8] (-0.45,-0.05);
  \draw[thick,---->-] (-0.95,-0.054) to[out=-45,in=175] (-0.25,-0.4);
  \draw[thick,---->-] (-1.195,-0.205) to[out=-45,in=175] (-0.29,-0.63);
  \draw[thick,----->-] (-1.4,-0.41) to[out=-45,in=173] (-0.33,-0.88) to[out=-7,in=180] (-0.06,-0.9) to[out=0,in=190] (0.27,-0.88);
  \draw[thick,--->-] (-0.29,-0.63) to (-0.25,-0.4);
  \draw[thick,--->-] (-0.25,-0.4) to (-0.02,-0.41);
  \draw[thick,--->-] (-0.29,-0.63) to (-0.04,-0.65);
  \draw[thick,--->-] (-0.04,-0.64) to (-0.02,-0.41);
  \draw[thick] (0.27,-0.88) to[out=280,in=190,looseness=0.9] (0.73,-1.4);
  \draw[thick,dashed] (0.73,-1.4) to[out=20,in=20,looseness=0.8] (0.35,-0.098);
  \draw[thick,--->-] (0.35,-0.098) to[out=200,in=100,looseness=0.9] (0.21,-0.4);
  \draw[thick,--->-] (-0.02,-0.41) to (0.21,-0.4);
  \draw[thick,->] (0.21,-0.4) to[out=10,in=250,looseness=0.95] (1,0);
  \draw[thick,->] (-0.04,-0.65) to[out=0
  ,in=250,looseness=0.9] (1.24,-0.1);
  \draw[thick,->] (0.27,-0.88) to[out=10,in=250] (1.48,-0.2);
  \draw[thick,--->-] (-1.05,0.4) to[out=250,in=135,looseness=1] (-0.95,-0.054);
  \draw[thick,---->-] (-1.3,0.45) to[out=250,in=135] (-1.195,-0.205);
  \draw[thick,---->-] (-1.55,0.5) to[out=250,in=135] (-1.4,-0.41);
  \fill (-0.25,-0.4) circle (0.05);
  \fill (-0.29,-0.63) circle (0.055);
  \fill[white] (-0.29,-0.63) circle (0.04);
  \fill (-0.02,-0.41) circle (0.055);
  \fill[white] (-0.02,-0.41) circle (0.04);
  \fill (-0.04,-0.65) circle (0.05);
  \fill (0.27,-0.88) circle (0.055);
  \fill[white] (0.27,-0.88) circle (0.04);
  \fill (0.21,-0.4) circle (0.05);
  \fill[red] (-0.95,-0.054) circle (0.05);
  \fill[red] (-1.195,-0.205) circle (0.05);
  \fill[red] (-1.4,-0.41) circle (0.05);
  \draw[red,below] (-0.95,-0.09) node {$1$};
  \draw[red,below] (-1.2,-0.24) node {$2$};
  \draw[red,below] (-1.4,-0.44) node {$3$};
  \draw[ultra thick, loosely dotted] (1,0.7) to[out=125,in=25,looseness=0.9] (-0.3,1);
  \draw[ultra thick] (0,1.5) to[out=180,in=90] (-2,0) to[out=270,in=180] (0,-1.5) to[out=0,in=270] (2,0) to[out=90,in=0] (0,1.5);
  \draw[ultra thick] (-0.5,0) to[out=50,in=130] (0.5,0);
  \draw[ultra thick] (-0.7,0.2) to[out=-60,in=240] (0.7,0.2);
\end{tikzpicture}
\caption{Network on a torus with a fixed cut}
\label{fig:TorusNetwork}
\end{figure}
With $\Gamma$ we can associate a network on a torus obtained by gluing the two boundary components of a cylinder matching the $N$ points as shown on Figure \ref{fig:TorusNetwork}. Let
\begin{align}
\mathcal C=\mathcal C^\bullet=\bigoplus_{i=0}^{+\infty}\mathcal C^{(i)}
\label{eq:CatTorus}
\end{align}
stands for the category of paths in the torus network starting/terminating at one of the $N$ points. Note that $\mathcal C$ is freely generated by $\mathcal C^{(1)}$, hence we can extend the double bracket on $\mathcal C^{(1)}$ to\footnote{Note that this double bracket is closely related but different from the double bracket on a conjugate surface corresponding to the torus network. The main difference is that we no longer think of objects of $\mathcal C$ as boundary points, but rather as points on the fixed cut of a conjugate surface. This records an additional information about the cut of the conjugate surface.
}
\begin{align*}
\db{\-,\-}:\quad\mathcal C\otimes\mathcal C\rightarrow\mathcal C\otimes\mathcal C.
\end{align*}

In this section we will prove that traces of powers of boundary measurement matrix $B(\lambda)$ serve as generating functions of hamiltonians in involution. Our proof is largely inspired by the well-known Lax method in classical theory of Integrable Systems \cite{Lax'1968,Calogero'1971,Moser'1975,OlshanetskyPerelomov'1976a,Dubrovin'1977, ReynmanSemenovTyanShanskii'1989} (See \cite{ReymanSemenov-Tian-Shansky'1994} for a review). The main difference with the commutative case is that we have three levels of noncommutative Poisson brackets on $\mathcal C$:
\begin{itemize}
\item The double bracket
\begin{align*}
\db{-,-}:\mathcal C\otimes\mathcal C\rightarrow\mathcal C\otimes\mathcal C,
\end{align*}
\item An $H_0$-Poisson structure
\begin{align*}
\{-,-\}:\mathcal C_\natural\otimes\mathcal C\rightarrow\mathcal C,
\end{align*}
\item Lie Bracket
\begin{align*}
\langle-,-\rangle:\mathcal C_\natural\otimes\mathcal C_\natural\rightarrow\mathcal C_\natural,
\end{align*}
\end{itemize}
where $\mathcal C_\natural$ stands for the \emph{universal trace} in $\mathcal C$: the linear span of
loops (i.e. directed cycles) modulo commutators. This is also called the \emph{cyclic space}, since it can be thought of
as words in the generators which represent cycles, but considered equivalent up to cyclic permutation.
It is worth noting that $C_\natural$ is considered as a vector space, and not an algebra.
The $H_0$-Poisson structure and Lie bracket mentioned above are induced by the double bracket, namely
\begin{align}
\{-,-\} = \mu \circ \db{-,-} \qquad \textrm{and} \qquad \langle-,-\rangle = \overline{\mu \circ \db{-,-}}.
\label{eq:InducedBrackets}
\end{align}
where $\mu$ is the concatenation map
\begin{align*}
\mu:\quad\mathrm{Hom}(A,B)\otimes\mathrm{Hom}(B,C)\rightarrow\mathrm{Hom}(A,C).
\end{align*}
Note that the right hand side in both formulas (\ref{eq:InducedBrackets}) is well-defined only when the first argument is a linear combination of loops. Since the bracket $\{-,-\}$ has the property that $\{ab,c\} = \{ba,c\}$ whenever $ab$ is a loop, the left
argument may be considered an element of $C_\natural$, as indicated above.

The Leibniz identity in both arguments is satisfied only on the level of a double bracket, while the two major statements of Theorems \ref{th:LaxEquation} and \ref{th:CommutativityOfHamiltonians} hold only at the level of $H_0$-Poisson structure and Lie Bracket respectively. In our proofs we will be working from the top to the bottom, first applying the Leibniz identity followed by considering an induced bracket.

At the last step of the proof it will be important for us that the trigonometric $r$-matrix (\ref{eq:TrigonometricRMatrix}) has the following form
\begin{align}
r(\lambda,\mu)=\sum_{m,n=1}^N\rho_{mn}(\lambda,\mu)E_{mn}\boxtimes E_{nm},
\label{eq:RMatrixAnsatz}
\end{align}
where $\rho_{mn}(\lambda,\mu)\in\mathbf k=\mathbb C(\lambda,\mu)$ are elements of the ground field.
\begin{remark}
Note that the twisted commutator formula (\ref{eq:twisted_commutator_formula}) combined with (\ref{eq:DoubleBracketSourcesAndTargets}) imposes severe restrictions on the form of an $r$-matrix, which come from examination of sources and targets of the resulting expression. In, particular, whenever $B(\lambda)$ is a square boundary measurement matrix with distinct sources and targets of all entries, the corresponding $r$-matrix entering (\ref{eq:twisted_commutator_formula}) must necessarily be of the form (\ref{eq:RMatrixAnsatz}).
\label{rem:TwistedCommutatorSourcesAndTargets}
\end{remark}

\begin{theorem}
Let $\Gamma$ be a perfect cylindrical network, $\db{-,-}$ stands for the double bracket on $\mathcal{C}$, and let
\begin{align*}
\{-,-\}:\mathcal{C}_\natural\otimes\mathcal C\rightarrow\mathcal C
\end{align*}
be the induced $H_0$-Poisson bracket. Then boundary measurement matrix $B(\lambda)$ for network $\Gamma$ satisfies the Lax equation
\begin{subequations}
\begin{align}
\left\{\overline{\mathrm{tr}\,B(\lambda)^k},B(\mu)\right\} =M(\lambda,\mu;k)B(\mu)-B(\mu)M(\lambda,\mu;k),
\end{align}
where
\begin{align}
M(\lambda,\mu;k)_{b_1b_2}=k\rho_{b_2b_1}(\lambda,\mu) \left(B(\lambda)^k\right)_{b_1b_2}.
\end{align}
\label{eq:NCLaxEquation}
\end{subequations}
\label{th:LaxEquation}
\end{theorem}
\begin{proof}
In order to avoid cumbersome intermediate formulas we employ Sweedler notation for the $r$-matrix, $r(\lambda,\mu)=r_1\boxtimes r_2$ and omit the summation index
over the products of pure tensors. Moreover, below we assume summation over repeating matrix indices, i.e.
\begin{align*}
X_{a_1a_2}Y_{a_2a_3}:=\sum_{a_2=1}^NX_{a_1a_2}Y_{a_2a_3}.
\end{align*}
Throughout the proof, the tensor product always stands for the usual tensor product of morphisms in $k$-linear category and concatenation
is applied componentwise $(f\otimes g)(h\otimes m)=(f h)\otimes(g m)$.

\underline{\textit{Step I:}} In terms of components, the twisted commutator bracket (\ref{eq:twisted_commutator_formula}) reads:
\begin{equation}
\begin{aligned}
    \db{B(\lambda)_{a_1a_3},B(\mu)_{b_1b_3}} =& \Big(\big[r(\lambda,\mu),B(\lambda)_L\boxtimes B(\mu)_R\big]_\tau\Big)_{(a_1a_3),(b_1b_3)}\\
    =& \big[(r_1)_{a_1a_2}\otimes (r_2)_{b_1b_2}\big] \big[B(\lambda)_{a_2a_3}\otimes B(\mu)_{b_2b_3}\big]\\
    & -\big[B(\mu)_{b_1b_2}\otimes B(\lambda)_{a_1a_2}\big] \big[(r_2)_{b_2b_3}\otimes (r_1)_{a_2a_3}\big]
\end{aligned}
\label{eq:RMatrixBracketIndicies}
\end{equation}

\underline{\textit{Step II:}} Compute the double bracket with a $k$-th power of a Lax matrix using double Leibniz identities
\begin{equation}
\begin{aligned}
\db{\big(B(\lambda)^k\big)_{a_1a_5},B(\mu)_{b_1b_5}} \stackrel{(\ref{eq:MatrixDoubleLeibniz})}{=}\;& \sum_{i=0}^k\left[\mathbf{1}_{b_1}\otimes \left(B(\lambda)^i\right)_{a_1a_2}\right] \db{B(\lambda)_{a_2a_4},B(\mu)_{b_1b_5}} \left[\left(B(\lambda)^{k-i-1}\right)_{a_4a_5}\otimes\mathbf{1}_{b_5}\right]\\
\stackrel{(\ref{eq:RMatrixBracketIndicies})}{=}\;&\sum_{i=0}^{k-1}\left[\mathbf{1}_{b_1}\otimes \left(B(\lambda)^i\right)_{a_1a_2}\right]\Big(\left[(r_1)_{a_2a_3}\otimes (r_2)_{b_1b_3}\right]\left[B(\lambda)_{a_3a_4}\otimes B(\mu)_{b_3b_5}\right]\\
&-\left[B(\mu)_{b_1b_3}\otimes B(\lambda)_{a_2a_3}\right] \left[(r_2)_{b_3b_5}\otimes(r_1)_{a_3a_4}\right]\Big) \left[\left(B(\lambda)^{k-i-1}\right)_{a_4a_5}\otimes\mathbf{1}_{b_5}\right]\\
=\;&\sum_{i=0}^{k-1}\left(r_1B(\lambda)^{k-i}\right)_{a_2a_5} \otimes \left(B(\lambda)^i\right)_{a_1a_2}\left(r_2B(\mu)\right)_{b_1b_5}\\
&-\sum_{i=0}^{k-1}(B(\mu)r_2)_{b_1b_5} \left(B(\lambda)^{k-i-1}\right)_{a_4a_5}\otimes \left(B(\lambda)^{i+1}r_1\right)_{a_1a_4}.
\end{aligned}
\label{eq:DoubleBracketPowerK}
\end{equation}

\underline{\textit{Step III:}} Compute an $H_0$-Poisson bracket with the trace of a $k$-th power of Lax matrix by taking the composition of (\ref{eq:DoubleBracketPowerK}) with the multiplication map.
\begin{equation}
\begin{aligned}
\left\{\overline{\left(B(\lambda)^k\right)_{a_1a_1}},B(\mu)_{b_1b_5}\right\} \;=\;&\mu\left(\db{\big(B(\lambda)^k\big)_{a_1a_1},B(\mu)_{b_1b_5}} \right)\\
\stackrel{(\ref{eq:DoubleBracketPowerK})}{=}& k\left(r_1B(\lambda)^k\right)_{a_2a_2} \left(r_2B(\mu)\right)_{b_1b_5}-k\left(B(\mu)r_2\right)_{b_1b_5} \left(B(\lambda)^kr_1\right)_{a_4a_4}\\
\stackrel{(\ref{eq:RMatrixAnsatz})}{=}&k\sum_{i,j=0}^N\rho_{ij}(\lambda,\mu) \left(E_{ij}B(\lambda)^k\right)_{a_2a_2} \left(E_{ji}B(\mu)\right)_{b_1b_5}\\
&-k\sum_{i,j=0}^N\rho_{ij}(\lambda,\mu)\left(B(\mu)E_{ji}\right)_{b_1b_5} \left(B(\lambda)^kE_{ij}\right)_{a_4a_4}\\
=\;&k\sum_{i=1}^N\rho_{ib_1}(\lambda,\mu) \left(B(\lambda)^k\right)_{b_1i}B(\mu)_{ib_5} -k\sum_{j=1}^N\rho_{b_5j}(\lambda,\mu)B(\mu)_{b_1j} \left(B(\lambda)^k\right)_{jb_5}\\
=\;&\big(M(\lambda,\mu;k)B(\mu)-B(\mu)M(\lambda,\mu;k)\big)_{b_1b_5}.
\end{aligned}
\label{eq:LaxEquationProof}
\end{equation}
\end{proof}

Similarly to the commutative case, Lax equation (\ref{eq:NCLaxEquation}) allows one to prove that Hamiltonians of the system Poisson commute with each other. However, it is worth noting that neither the double bracket, nor the $H_0$-Poisson bracket between traces of powers of boundary measurement matrix vanish identically, only the Lie
bracket between the associated elements of the cyclic space does.

\begin{theorem}
\label{th:CommutativityOfHamiltonians}
Let $\Gamma$ be a perfect cylindrical network, $\db{-,-}$ the double bracket on $\mathcal{C}$, and let $\left<-,-\right>$
be the induced bracket on the cyclic space $\mathcal{C}_\natural$.
Define the elements $H_{ij} \in \mathcal{C}$ by
\[ \mathrm{tr} \left( B_\Gamma(\lambda)^i \right) = \sum_j H_{ij} \lambda^j \]
Then $\left<\,\overline{H_{ij}},\, \overline{H_{k\ell}}\,\right> = 0$ for all $i,j,k,\ell$.
\end{theorem}
\begin{proof}
An $H_0$-Poisson bracket $\{-,-\}:\mathcal C_\natural\otimes\mathcal C\rightarrow\mathcal C$ satisfies Leibniz identity in the second argument, i.e. for all $\quad h\in\mathcal C_\natural$ and all composable $f,g\in\mathcal C$ we have
\begin{align}
\{h,fg\}=f\{h,g\}+\{h,f\}g.
\label{eq:LeibnizIdentityH0}
\end{align}

\underline{\textit{Step I:}} Compute an $H_0$-Poisson bracket between traces of powers of $B(\lambda)$
\begin{equation}
\begin{aligned}
\left\{\overline{\left(B(\lambda)^k\right)_{a_1a_1}},\left(B(\mu)^l\right)_{b_1b_5}\right\} \;\stackrel{(\ref{eq:LeibnizIdentityH0})}{=}&\sum_{j=0}^{l-1}\left(B(\mu)^j\right)_{b_1b_2} \left\{\overline{\left(B(\lambda)^k\right)_{a_1a_1}},\left(B(\mu)^l\right)_{b_2b_4}\right\} \left(B(\mu)^{l-j-1}\right)_{b_4b_5}\\
\stackrel{(\ref{eq:LaxEquationProof})}{=}&\sum_{j=0}^{l-1}\left( B(\mu)^jM(\lambda,\mu;k)B(\mu)^{l-j} -B(\mu)^{j+1}M(\lambda,\mu;k)B(\mu)^{l-j-1}\right)_{b_1b_5}\\
=\;&\left(M(\lambda,\mu;k)B(\mu)^l-B(\mu)^lM(\lambda,\mu;k)\right)_{b_1b_5}.
\end{aligned}
\label{eq:H0PowersLaxProof}
\end{equation}

\underline{\textit{Step II:}} Taking the the trace and quotient by the commutant on both sides of (\ref{eq:H0PowersLaxProof}) we obtain
\begin{align*}
\left\langle\overline{\left(B(\lambda)^k\right)_{a_1a_1}}, \overline{\left(B(\mu)^l\right)_{b_1b_1}}\right\rangle\;=\;\;&\overline{\left\{\overline{ \left(B(\lambda)^k\right)_{a_1a_1}},\left(B(\mu)^l\right)_{b_1b_1}\right\}}\\
\stackrel{(\ref{eq:H0PowersLaxProof})}{=}\;& \overline{\left(M(\lambda,\mu;k)B(\mu)^l-B(\mu)^lM(\lambda,\mu;k)\right)_{b_1b_1}}\\[5pt]
=\;\;&\overline{M(\lambda,\mu;k)_{b_1b_2}\left(B(\mu)^l\right)_{b_2b_1}} -\overline{\left(B(\mu)^l\right)_{b_2b_1}M(\lambda,\mu;k)_{b_1b_2}}\\[5pt]
=\;\;&0\bmod[\mathcal C,\mathcal C].
\end{align*}
\end{proof}

\section{$r$-Matrix Brackets and Double Quasi Jacobi Identity}
\label{sec:DoubleQuasiJacobi}

In this section, we describe sufficient conditions for an $r$-matrix to define a double quasi Poisson bracket via twisted commutator formula (\ref{eq:TwistedCommutatorSpectralParameterGeneral}). For this section, let $\mathcal C$ be a small $\mathbf k$-linear category with $m+n$ objects divided into $m$ ``sources'' $p_1,\dots,p_m$ and $n$ ``targets'' $q_1,\dots,q_n$. Moreover, we assume that $\mathcal C$ is $\mathbb Z$-graded with (at most) one-dimensional homogeneous components
\begin{align*}
\dim_{\mathbf k} \mathrm{Hom}_k(p_i,q_j)\leqslant 1,\qquad 1\leq i\leq m,\quad 1\leq j\leq n.
\end{align*}
We can introduce the formal $m\times n$ Lax matrix $B(\lambda)$ with a spectral parameter by describing its matrix elements as
\begin{align*}
B(\lambda)_{i,j}=\sum_{k\in\mathbb Z} h_{i,j}(k) \lambda^k\quad\in\quad\mathrm{Hom}_k(p_i,q_j),\qquad 1\leq i\leq m,\quad 1\leq j\leq n.
\end{align*}
Here $h_{i,j}(k)\in\mathrm{Hom}_k(p_i,q_j)$ stands for the only generator of the corresponding homogeneous component of $\mathrm{Hom}(p_i,q_j)$.

Define biderivation on $\mathcal C$ by the twisted commutator formula (\ref{eq:TwistedCommutatorSpectralParameterGeneral})
\begin{equation}
\begin{aligned}
    \db{B(\lambda),B(\mu)} \, = r(\lambda,\mu) (B_L(\lambda) \boxtimes B_R(\mu)) - \Big( (B_L(\lambda) \boxtimes B_R(\mu) \, \underline r(\lambda,\mu) \Big)^\tau,
\end{aligned}
\label{eq:TwistedCommutatorFormal}
\end{equation}
where $r$ and $\underline r$ are $m^2\times m^2$ and $n^2\times n^2$ matrices respectively, with entries in $\mathbf k(\lambda,\mu)$:
\begin{align*}
r=&\sum_{i,j,k,l=1}^m r_{ij}^{kl}(\lambda,\mu)\, E(m)_{ij}\boxtimes E(m)_{kl},&
\underline r=&\sum_{i,j,k,l=1}^n \underline r_{ij}^{kl}(\lambda,\mu)\, E(n)_{ij}\boxtimes E(n)_{kl}.
\end{align*}
The double brackets between Lax matrix elements read
\begin{align*}
\db{B(\lambda)_{a_1a_3},B(\mu)_{b_1b_3}}=\sum_{a_2,b_2=1}^m r_{a_1a_2}^{b_1b_2}(\lambda,\mu)\, (B(\lambda)_{a_2a_3}\otimes B(\mu)_{b_2b_3}) -\sum_{a_2,b_2=1}^n(B(\mu)_{b_1b_2}\otimes B(\lambda)_{a_1a_2})\, \underline r_{a_2a_3}^{b_2b_3}(\lambda,\mu).
\end{align*}

\subsection*{Skew-Symmetry}

From (\ref{eq:DoubleBracketSkewSymmetry}) we know that double bracket (\ref{eq:TwistedCommutatorFormal}) should satisfy
\begin{align*}
0=&\db{B(\lambda)_{a_1a_3},B(\mu)_{b_1b_3}} +\db{B(\mu)_{b_1b_3},B(\lambda)_{a_1a_3}}^\tau \\
=&\sum_{a_2,b_2=1}^m\big(r_{a_1a_2}^{b_1b_2}(\lambda,\mu)+r_{b_1b_2}^{a_1a_2}(\mu,\lambda)\big)\, B(\lambda)_{a_2a_3} \otimes B_{b_2b_3}(\mu)-\sum_{a_2,b_2=1}^n\big(\underline r_{a_2a_3}^{b_2b_3}(\lambda,\mu)+\underline r_{b_2b_3}^{a_2a_3}(\mu,\lambda)\big)\, B(\mu)_{b_1b_2}\otimes B(\lambda)_{a_1a_2}
\end{align*}
As a corollary, sufficient conditions for an $r$-matrix to define skew-symmetric biderivation are
\begin{subequations}
\begin{align}
r_{a_1a_2}^{b_1b_2}(\lambda,\mu)+r_{b_1b_2}^{a_1a_2}(\mu,\lambda)=0,\\
\underline r_{a_2a_3}^{b_2b_3}(\lambda,\mu)+\underline r_{b_2b_3}^{a_2a_3}(\mu,\lambda)=0.
\end{align}
\label{eq:SkewSymmetryRMatrix}
\end{subequations}

\subsection*{Left Hand Side of Double Quasi-Jacobi Identity}

Biderivation (\ref{eq:TwistedCommutatorFormal}) defines a triple bracket $\db{,,}$ (\cite{VandenBergh'2008})
which appears on the left hand side of formula \ref{eq:DoubleQuasiJacobiIdentity}.
Action of the triple bracket on the Lax matrix entries can be computed by the following formula
\begin{align*}
\db{B(\lambda)_{a_1a_4},B(\mu)_{b_1b_4},B(\nu)_{c_1c_4}}=& \db{B(\lambda)_{a_1a_4},\db{B(\mu)_{b_1b_4},B(\nu)_{c_1c_4}}}_L +\sigma_{(1,2,3)}\big(\db{B(\mu)_{b_1b_4},\db{B(\nu)_{c_1c_4}, B(\lambda)_{a_1a_4}}}_L\big)\\
&+\sigma_{(1,3,2)}\big(\db{B(\nu)_{c_1c_4},\db{B(\lambda)_{a_1a_4},B(\mu)_{b_1b_4}}}_L\big),
\end{align*}
where $\db{x,y\otimes z}_L=\db{x,y}\otimes z$, and $\sigma_{(1,2,3)}$ is the cyclic permutation of tensor factors in $\mathcal{C}^{\otimes 3}$.
For each of the three terms we get
\begin{align*}
\db{B(\lambda)_{a_1a_4},\db{B(\mu)_{b_1b_4},B(\nu)_{c_1c_4}}}_L= &r_{b_1b_2}^{c_1c_2}(\mu,\nu)r_{a_1a_2}^{b_2b_3}(\lambda,\mu)\, B(\lambda)_{a_2a_4}\otimes B(\mu)_{b_3b_4}\otimes B(\nu)_{c_2c_4}\\
&-r_{b_1b_2}^{c_1c_2}(\mu,\nu)\underline r_{a_2a_4}^{b_3b_4}(\lambda,\mu)\, B(\mu)_{b_2b_3}\otimes B(\lambda)_{a_1a_2}\otimes B(\nu)_{c_2c_4}\\
&-r_{a_1a_2}^{c_1c_2}(\lambda,\nu)\underline r_{b_2b_4}^{c_3c_4}(\mu,\nu)\, B(\lambda)_{a_2a_4}\otimes B(\nu)_{c_2c_3}\otimes B(\mu)_{b_1b_2}\\
&+\underline r_{a_2a_4}^{c_2c_3}(\lambda,\nu)\underline r_{b_2b_4}^{c_3c_4}(\mu,\nu)\, B(\nu)_{c_1c_2}\otimes B(\lambda)_{a_1a_2}\otimes B(\mu)_{b_1b_2},
\end{align*}
\begin{align*}
\sigma_{(1,2,3)}\db{B(\mu)_{b_1b_4},\db{B(\nu)_{c_1c_4}, B(\lambda)_{a_1a_4}}}_L= &r_{c_1c_2}^{a_1a_2}(\nu,\lambda)r_{b_1b_2}^{c_2c_3}(\mu,\nu)\, B(\lambda)_{a_2a_4}\otimes B(\mu)_{b_2b_4}\otimes B(\nu)_{c_3c_4}\\
&-r_{c_1c_2}^{a_1a_2}(\nu,\lambda)\underline r_{b_2b_4}^{c_3c_4}(\mu,\nu)\, B(\lambda)_{a_2a_4}\otimes B(\nu)_{c_2c_3}\otimes B(\mu)_{b_1b_2}\\
&-r_{b_1b_2}^{a_1a_2}(\mu,\lambda)\underline r_{c_2c_4}^{a_3a_4}(\nu,\lambda)\, B(\nu)_{c_1c_2}\otimes B(\mu)_{b_2b_4}\otimes B(\lambda)_{a_2a_3},\\
&\underline r_{b_2b_4}^{a_2a_3}(\mu,\lambda)\underline r_{c_2c_4}^{a_3a_4}(\nu,\lambda)\, B(\nu)_{c_1c_2}\otimes B(\lambda)_{a_1a_2}\otimes B(\mu)_{b_1b_2},
\end{align*}
\begin{align*}
\sigma_{(1,3,2)}\db{B(\nu)_{c_1c_4},\db{B(\lambda)_{a_1a_4},B(\mu)_{b_1b_4}}}_L= &r_{a_1a_2}^{b_1b_2}(\lambda,\mu)r_{c_1c_2}^{a_2a_3}(\nu,\lambda)\, B(\lambda)_{a_3a_4}\otimes B(\mu)_{b_2b_4}\otimes B(\nu)_{c_2c_4}\\
&-r_{a_1a_2}^{b_1b_2}(\lambda,\mu)\underline r_{c_2c_4}^{a_3a_4}(\nu,\lambda)\, B(\nu)_{c_1c_2}\otimes B(\mu)_{b_2b_4}\otimes B(\lambda)_{a_2a_3}\\
&-r_{c_1c_2}^{b_1b_2}(\nu,\mu)r_{a_2a_4}^{b_3b_4}(\lambda,\mu)\, B(\mu)_{b_2b_3}\otimes B(\lambda)_{a_1a_2}\otimes B(\nu)_{c_2c_4}\\
&+\underline r_{c_2c_4}^{b_2b_3}(\nu,\mu)\underline r_{a_2a_4}^{b_3b_4}(\lambda,\mu)\, B(\nu)_{c_1c_2}\otimes B(\lambda)_{a_1a_2}\otimes B(\mu)_{b_1b_2}.
\end{align*}
Here and until the end of the subsection we assume the summation over repeating matrix indices. Now, combining coefficients in front of each tensor product we obtain
\begin{equation*}
\begin{aligned}
\db{B(\lambda)_{a_1a_4},B(\mu)_{b_1b_4},B(\nu)_{c_1c_4}}=&
\left(r_{b_1b_2}^{c_1c_3}(\mu,\nu)r_{a_1a_3}^{b_2b_3}(\lambda,\mu) +r_{c_1c_2}^{a_1a_3}(\nu,\lambda)r_{b_1b_3}^{c_2c_3}(\mu,\nu) +r_{a_1a_2}^{b_1b_3}(\lambda,\mu)r_{c_1c_3}^{a_2a_3}(\nu,\lambda)\right)
\\&\qquad\times
B(\lambda)_{a_3a_4}\otimes B(\mu)_{b_3b_4}\otimes B(\nu)_{c_3c_4}\\
&+\left(\underline r_{c_2c_3}^{a_2a_4}(\lambda,\nu)\underline r_{b_2b_4}^{c_3c_4}+\underline r_{b_2b_4}^{a_2a_3}(\mu,\lambda)\underline r_{c_2c_4}^{a_3a_4}(\nu,\lambda)+\underline r_{c_2c_4}^{b_2b_3}(\nu,\mu)\underline r_{a_2a_4}^{b_3b_4}(\lambda,\mu)\right)
\\&\qquad\times
B(\nu)_{c_1c_2}\otimes B(\lambda)_{a_1a_2}\otimes B(\mu)_{b_1b_2}\\
&-\left(r_{b_1b_2}^{c_1c_2}(\mu,\nu)\underline r_{a_2a_4}^{b_3b_4}(\lambda,\mu)+r_{c_1c_2}^{b_2b_3}(\nu,\mu)\underline r_{a_2a_n}^{b_3b_4}(\lambda,\mu)\right)\, B(\mu)_{b_2b_3}\otimes B(\lambda)_{a_1a_2}\otimes B(\nu)_{c_2c_4}\\
&-\left(r_{a_1a_2}^{c_1c_2}(\lambda,\nu)\underline r_{b_2b_4}^{c_3c_4}(\mu,\nu)+r_{c_1c_2}^{a_1a_2}(\nu,\lambda)\underline r_{b_2b_4}^{c_3c_4}(\mu,\nu)\right)\, B(\lambda)_{a_2a_4}\otimes B(\nu)_{c_2c_3}\otimes B(\mu)_{b_1b_2}\\
&-\left(r_{b_1b_2}^{a_1a_2}(\mu,\lambda)\underline r_{c_2c_4}^{a_3a_4}(\nu,\lambda)+r_{a_1a_2}^{b_1b_2}(\lambda,\mu)\underline r_{c_2c_4}^{a_3a_4}(\nu,\lambda)\right)\, B(\nu)_{c_1c_2}\otimes B(\mu)_{b_2b_4}\otimes B(\lambda)_{a_2a_3}.
\end{aligned}
\end{equation*}
Assuming that both $r$ and $\underline r$ satisfy skew-symmetry condition (\ref{eq:SkewSymmetryRMatrix}), the last three terms cancel in the formula above and we obtain
\begin{equation}
\begin{aligned}
\db{B&(\lambda)_{a_1a_4},B(\mu)_{b_1b_4},B(\nu)_{c_1c_4}}=\\
&
\quad\left(r_{b_1b_2}^{c_1c_3}(\mu,\nu)r_{a_1a_3}^{b_2b_3}(\lambda,\mu) +r_{c_1c_2}^{a_1a_3}(\nu,\lambda)r_{b_1b_3}^{c_2c_3}(\mu,\nu) +r_{a_1a_2}^{b_1b_3}(\lambda,\mu)r_{c_1c_3}^{a_2a_3}(\nu,\lambda)\right)
B(\lambda)_{a_3a_4}\otimes B(\mu)_{b_3b_4}\otimes B(\nu)_{c_3c_4}\\
&+\left(\underline r_{c_2c_3}^{a_2a_4}(\lambda,\nu)\underline r_{b_2b_4}^{c_3c_4}(\mu,\nu)+\underline r_{b_2b_4}^{a_2a_3}(\mu,\lambda)\underline r_{c_2c_4}^{a_3a_4}(\nu,\lambda)+\underline r_{c_2c_4}^{b_2b_3}(\nu,\mu)\underline r_{a_2a_4}^{b_3b_4}(\lambda,\mu)\right)
B(\nu)_{c_1c_2}\otimes B(\lambda)_{a_1a_2}\otimes B(\mu)_{b_1b_2}
\end{aligned}
\label{eq:JacobiLHSLax}
\end{equation}

\subsection*{Right Hand Side of Double Quasi-Jacobi Identity}

The trivector from the right hand side of the double Quasi Jacobi Identity (\ref{eq:DoubleQuasiJacobiIdentity}) reads:
\begin{align*}
\mathbf{T}=\frac14\sum_{V\in\mathrm{Obj}(\mathcal C)}\overline{\partial_V\ast\partial_V\ast\partial_V}=\sum_{i=1}^m\overline{\partial_{p_i} \ast\partial_{p_i} \ast\partial_{p_i}}+\sum_{j=1}^n\overline{\partial_{q_j} \ast\partial_{q_j} \ast\partial_{q_j}}
\end{align*}
From (\ref{eq:UniderivationV}) we get the action of uniderivations on matrix elements
\begin{align*}
\partial_{p_k}B(\lambda)_{ij}=&-\delta_{ki}(\mathbf 1_{p_i}\otimes B(\lambda)_{ij}),\\
\partial_{q_k}B(\lambda)_{ij}=&\delta_{kj}(B(\lambda)_{ij}\otimes\mathbf 1_{q_i}).
\end{align*}
As a corollary, we get
\begin{equation}
\begin{aligned}
\mathbf T(B(\lambda)_{a_1a_4}\otimes B(\mu)_{b_1b_4}\otimes B(\nu)_{c_1c_4})=&-\frac14\delta_{a_1b_1}\delta_{a_1c_1}B(\lambda)_{a_1a_4}\otimes B(\mu)_{b_1b_4}\otimes B(\nu)_{c_1c_4}\\
&+\frac14\delta_{a_4b_4}\delta_{a_4c_4} B(\nu)_{c_1c_4}\otimes B(\lambda)_{a_1a_4}\otimes B(\mu)_{b_1b_4}.
\end{aligned}
\label{eq:JacobiRHSLax}
\end{equation}

\subsection*{Quasi Yang-Baxter Equation}

Combining (\ref{eq:JacobiLHSLax}) with (\ref{eq:JacobiRHSLax}) we obtain sufficient conditions for $r$ and $\underline r$ which would guarantee that biderivation (\ref{eq:TwistedCommutatorFormal}) satisfies double Quasi Jacobi Identity
\begin{subequations}
\begin{equation}
\begin{aligned}
\sum_{b_2=1}^m r_{b_1b_2}^{c_1c_3}(\mu,\nu)r_{a_1a_3}^{b_2b_3}(\lambda,\mu) +\sum_{c_2=1}^m r_{c_1c_2}^{a_1a_3}(\nu,\lambda)r_{b_1b_3}^{c_2c_3}(\mu,\nu) +\sum_{a_2=1}^m r_{a_1a_2}^{b_1b_3}(\lambda,\mu)r_{c_1c_3}^{a_2a_3}(\nu,\lambda)&\\
=- \frac14\delta_{a_1b_1}\delta_{a_1c_1}&\delta_{a_1a_3}\delta_{b_1b_3}\delta_{c_1c_3},
\end{aligned}
\label{eq:QuasiYangBaxterI}
\end{equation}
\begin{equation}
\begin{aligned}
\sum_{c_3=1}^n\underline r_{c_2c_3}^{a_2a_4}(\lambda,\nu)\underline r_{b_2b_4}^{c_3c_4}(\mu,\nu)+\sum_{a_3=1}^n\underline r_{b_2b_4}^{a_2a_3}(\mu,\lambda)\underline r_{c_2c_4}^{a_3a_4}(\nu,\lambda)+\sum_{b_3=1}^n\underline r_{c_2c_4}^{b_2b_3}(\nu,\mu)\underline r_{a_2a_4}^{b_3b_4}(\lambda,\mu)&\\
=\frac14\delta_{a_4b_4}\delta_{a_4c_4}& \delta_{a_2a_4}\delta_{b_2b_4}\delta_{c_2c_4}
\label{eq:QuasiYangBaxterII}
\end{aligned}
\end{equation}
\label{eq:QuasiYangBaxter}
\end{subequations}
Note that equation (\ref{eq:QuasiYangBaxterII}) has exactly the same form as the equation (\ref{eq:QuasiYangBaxterI}) modulo relabelling indices and skew symmetry relations (\ref{eq:SkewSymmetryRMatrix}).

\begin{proposition}
Trigonometric $r$-matrix (\ref{eq:TrigonometricRMatrix}) satisfies skew-symmetry identity (\ref{eq:SkewSymmetryRMatrix}) and identity (\ref{eq:QuasiYangBaxter}).
\end{proposition}
\begin{proof}
Recall that matrix elements of the trigonometric $r$-matrix (\ref{eq:TrigonometricRMatrix}) have the following form
\begin{align}
r_{a_1a_2}^{b_1b_2}(\lambda,\mu)=\rho_{a_1a_2}(\lambda,\mu)\delta_{a_1b_2}\delta_{a_2b_1}, \qquad \rho_{a_1a_2}(\lambda,\mu)=\left\{\begin{array}{ll}
\frac\lambda{\mu-\lambda}&a_1<a_2,\\[5pt]
\frac12\frac{\mu+\lambda}{\mu-\lambda}&a_1=a_2,\\[5pt]
\frac\mu{\mu-\lambda}&a_1>a_2.
\end{array}\right.
\label{eq:TrigonometricRMatrixMatrixElements}
\end{align}
From (\ref{eq:TrigonometricRMatrix}) it is immediate that trigonometric $r$-matrix satisfies skew-symmetry identity (\ref{eq:SkewSymmetryRMatrix}). At the same time, equation (\ref{eq:QuasiYangBaxterI}) acquires the following form
\begin{align*}
\sum_{b_2=1}^m\rho_{b_1b_2}(\mu,\nu)\delta_{b_1c_3}\delta_{b_2c_1} \rho_{a_1a_3}(\lambda,\mu)\delta_{a_1b_3}\delta_{a_3b_2} +\sum_{c_2=1}^m\rho_{c_1c_2}(\nu,\lambda)\delta_{c_1a_3}\delta_{c_2a_1} \rho_{b_1b_3}(\mu,\nu)\delta_{b_1c_3}\delta_{b_3c_2}\\ +\sum_{a_2=1}^m\rho_{a_1a_2}(\lambda,\mu)\delta_{a_1b_3}\delta_{a_2b_1} \rho_{c_1c_3}(\nu,\lambda)\delta_{c_1a_3}\delta_{c_3a_2}=-\frac14 \delta_{a_1b_1}\delta_{a_1c_1}&\delta_{a_1a_3}\delta_{b_1b_3}\delta_{c_1c_3}
\end{align*}
Taking the sum and factoring out $\delta_{b_1c_3}\delta_{a_1b_3}\delta_{a_3c_1}$ we obtain an equivalent equation for $\rho$:
\begin{align}
\rho_{b_1c_1}(\mu,\nu)\rho_{a_1c_1}(\lambda,\mu) +\rho_{c_1a_1}(\nu,\lambda)\rho_{b_1a_1}(\mu,\nu) +\rho_{a_1b_1}(\lambda,\mu)\rho_{c_1b_1}(\nu,\lambda) +\frac14\delta_{a_1b_1}\delta_{a_1c_1}=0.
\label{eq:QuasiYangBaxterRho}
\end{align}

Note that equation (\ref{eq:QuasiYangBaxterRho}) is invariant under the simultaneous cyclic permutation of $(a_1,b_1,c_1)$ and $(\lambda,\mu,\nu)$, so without loss of generality we can assume that $a_1$ is always the smallest index. This leaves us with six cases
\begin{enumerate}[{Case} 1:]
\item When $a_1<b_1<c_1$, substituting (\ref{eq:TrigonometricRMatrixMatrixElements}) to (\ref{eq:QuasiYangBaxterRho}) we get
\begin{align*}
\frac{\mu}{\nu-\mu}\frac{\lambda}{\mu-\lambda} +\frac{\lambda}{\lambda-\nu}\frac{\nu}{\nu-\mu} +\frac{\lambda}{\mu-\lambda}\frac{\lambda}{\lambda-\nu}=0.
\end{align*}
\item When $a_1<c_1<b_1$ we have
\begin{align*}
\frac{\nu}{\nu-\mu}\frac{\lambda}{\mu-\lambda} +\frac{\lambda}{\lambda-\nu}\frac{\nu}{\nu-\mu} +\frac{\lambda}{\mu-\lambda}\frac{\nu}{\lambda-\nu}=0.
\end{align*}
\item When $a_1=b_1<c_1$
\begin{align*}
\frac{\mu}{\nu-\mu}\frac{\lambda}{\mu-\lambda} +\frac{\lambda}{\lambda-\nu}\,\frac12\frac{\mu+\nu}{\nu-\mu} +\frac12\frac{\lambda+\mu}{\mu-\lambda}\frac{\lambda}{\lambda-\nu}=0.
\end{align*}
\item When $a_1<b_1=c_1$
\begin{align*}
\frac12\frac{\nu+\mu}{\nu-\mu}\frac{\lambda}{\mu-\lambda} +\frac{\lambda}{\lambda-\nu}\frac{\nu}{\nu-\mu} +\frac{\lambda}{\mu-\lambda}\,\frac12\frac{\lambda+\nu}{\lambda-\nu}=0
\end{align*}
\item When $a_1=c_1<b_1$
\begin{align*}
\frac\nu{\nu-\mu}\,\frac12\frac{\mu+\lambda}{\mu-\lambda} +\frac12\frac{\lambda+\nu}{\lambda-\nu}\frac\nu{\nu-\mu} +\frac\lambda{\mu-\lambda}\frac\nu{\lambda-\nu}=0
\end{align*}
\item Finally, when $a_1=b_1=c_1$ we have
\begin{align*}
\frac12\frac{\mu+\nu}{\nu-\mu}\,\frac12\frac{\mu+\lambda}{\mu-\lambda} +\frac12\frac{\lambda+\nu}{\lambda-\nu}\,\frac12\frac{\nu+\mu}{\nu-\mu} +\frac12\frac{\mu+\lambda}{\mu-\lambda}\,\frac12\frac{\lambda+\nu}{\lambda-\nu}+\frac14=0.
\end{align*}
\end{enumerate}
\end{proof}

\section{Refactorization Dynamics}

In the theory of classical Integrable Systems, Lax equation with $r$-matrix (\ref{eq:RMatrixDisc}) allows one to define a Hamilton flow which interpolates the refactorization of Lax matrix $B$ as a product of lower and upper triangular matrices \cite{SemenovTianShansky'1985} (see also \cite{ReymanSemenov-Tian-Shansky'1994} and \cite{HoffmannKellendonkKutzReshetikhin'2000}). In this section we show that some of the classical results about refactorization dynamics generalize to noncommutative case.

Let $\Gamma$ be a perfect cylindrical network with equal number of inputs and outputs, say $N$. As in Section \ref{sec:LaxEquationTraceHomiltonians}, consider a torus glued out of cylinder by matching the marked points and denote by $\mathcal C$ the associated category of paths starting/terminating at one of the $N$ marked points (see Figure \ref{fig:TorusNetwork}). Denote the corresponding boundary measurement matrix by $B(\lambda)$.

In the classical case, interpolating hamiltonian for refactorization dynamics is given by
\begin{align}
H=\frac12\mathrm{tr}\left(\log B\right)^2.
\label{eq:ClassicalInterpolatingHamiltonian}
\end{align}
For the purpose of this section we will restrict our attention to the formal neighbourhood of the identity matrix. Then one can think of (\ref{eq:ClassicalInterpolatingHamiltonian}) as a power series in $(B-\mathbf 1)$.

\subsection{Brackets Between Infinite Series}

Our first major goal is compute a Hamilton flow given by (\ref{eq:ClassicalInterpolatingHamiltonian}) in the noncommutative case. To this end we must extend double Poisson brackets (\ref{eq:twisted_commutator_formula}) between matrix elements of the boundary measurement matrix $B(\lambda)$ to an appropriate completion. Indeed, let
\begin{align}
X(\lambda)=B(\lambda)-\mathbf 1
\label{eq:ShiftedBoundaryL}
\end{align}
be a difference with the identity matrix. Because there is no relations between matrix elements of $X$ we can introduce a grading on $\mathcal C$ by nonnegative integers
\begin{align*}
\mathcal C=\bigoplus_{m\in\mathbb Z_{\geqslant0}}\mathcal C_{(m)},
\end{align*}
where $\mathcal C_{(m)}$ is spanned by monomials of length $m$ in matrix elements $X_{i,j}$. Here we allow only composable monomials such as $X_{1,2}X_{2,4}X_{4,7}$.

Note that the above grading is slightly different from grading (\ref{eq:CatTorus}) by the winding number due to the fact that $X(\lambda)$ is shifted by identity matrix (\ref{eq:ShiftedBoundaryL}). In particular, the double bracket is no longer homogeneous with respect to the above grading. Instead, we have the following
\begin{lemma}
The restriction of the double bracket on homogeneous components is given by a map
\begin{align}
\db{\-,\-}:\quad \mathcal C_{(m)}\otimes\mathcal C_{(n)}\rightarrow \bigoplus_{\epsilon=0}^2\;\left(\mathcal C\otimes\mathcal C\right)_{m+n-\epsilon}.
\label{eq:LDegreeBoundDoubleBracket}
\end{align}
\end{lemma}
\begin{proof}
The double bracket between matrix elements of $X(\lambda)$ reads
\begin{align*}
\db{X(\lambda)_{a_1a_3},X(\mu)_{b_1b_3}}\;\;=\;\;&\db{B(\lambda)_{a_1a_3},B(\mu)_{b_1b_3}}\\
\stackrel{(\ref{eq:RMatrixBracketIndicies})}{=}\;
&\big[(r_1)_{a_1a_2}\otimes (r_2)_{b_1b_2}\big] \big[\left(X(\lambda)+\mathbf{1}\right)_{a_2a_3}\otimes \left(X(\mu)+\mathbf{1}\right)_{b_2b_3}\big]\\
    & -\big[\left(X(\mu)+\mathbf{1}\right)_{b_1b_2}\otimes \left(X(\lambda)+\mathbf{1}\right)_{a_1a_2}\big] \big[(r_2)_{b_2b_3}\otimes (r_1)_{a_2a_3}\big].
\end{align*}
As a corollary,
\begin{align}
\db{X(\lambda)_{a_1a_3},X(\mu)_{b_1b_3}}\quad\in\; \bigoplus_{\epsilon_1,\epsilon_2=0,1}\left(\mathcal C_{(1-\epsilon_1)}\otimes\mathcal C_{(1-\epsilon_2)}\right)\quad\subset\quad \bigoplus_{\epsilon=0}^2\;\left(\mathcal C\otimes\mathcal C\right)_{(2-\epsilon)}
\label{eq:LGradingBracketBetweenGenerators}
\end{align}
Combining (\ref{eq:LGradingBracketBetweenGenerators}) with the double Leibniz identities (\ref{eq:DoubleLeibnizIdentityI})--(\ref{eq:DoubleLeibnizIdentityII}) we get the statement of the Lemma.
\end{proof}

Now let $\hat{\mathcal C}$ stand for the $\mathbf k$-linear space of formal power series in matrix elements of $X$
\begin{align*}
\hat{\mathcal C}=\prod_{m\in\mathbb Z_{\geqslant0}}\mathcal C_{(m)}.
\end{align*}
Because the degree of monomials in the double bracket between homogeneous elements is bounded from below by (\ref{eq:LDegreeBoundDoubleBracket}) we can extend the double bracket on pairs of series
\begin{align*}
\db{\-,\-}:\;\hat{\mathcal C}\otimes\hat{\mathcal C}\rightarrow\hat{\mathcal C}\otimes\hat{\mathcal C}.
\end{align*}
Moreover, formula (\ref{eq:LDegreeBoundDoubleBracket}) ensures that the induced brackets can be extended to an appropriate completion as well
\begin{align*}
\{\-,\-\}:&\;\hat{\mathcal C}_\natural\otimes\hat{\mathcal C}\rightarrow\hat{\mathcal C}, \qquad\textrm{where}\qquad \{-,-\} = \mu \circ \db{-,-},\\[5pt]
\langle\-,\-\rangle:&\;\hat{\mathcal C}_\natural\otimes\hat{\mathcal C}_\natural\rightarrow\hat{\mathcal C}_\natural,\qquad \textrm{where} \qquad \langle-,-\rangle = \overline{\mu \circ \db{-,-}},
\end{align*}
where $\hat{\mathcal C}_\natural$ stands for the completion of the cyclic space
\begin{align*}
\hat{\mathcal C}_\natural=\prod_{m\in\mathbb Z_{\geqslant0}}\big({\mathcal C}_\natural\big)_{(m)}.
\end{align*}

\subsection{Noncommutative Refactorization Flow}

We are now ready to calculate the Hamilton flow given by noncommutative analogue of (\ref{eq:ClassicalInterpolatingHamiltonian}). We will be interested in the action of this Hamilton flow on the boundary measurement matrix $B(\lambda)$. As in commutative case \cite{ReymanSemenov-Tian-Shansky'1994}, the central role in the calculation is played by Lax equation (\ref{eq:NCLaxEquation}).

\begin{proposition}
Consider a hamiltonian
\begin{align*}
H(\lambda)=\frac12\,\overline{\mathrm{tr}\left(\log\,B(\lambda)\right)^2}
\qquad\in\quad\hat{\mathcal C}_\natural,
\end{align*}
then
\begin{align*}
\left\{H(\lambda),B(\mu)\right\}=\widetilde{M}(\lambda,\mu)B(\mu) -B(\mu)\widetilde{M}(\lambda,\mu),
\qquad\textrm{where}\quad
\widetilde{M}(\lambda,\mu)_{b_1b_2}=\rho_{b_2b_1}(\lambda,\mu)\left(\log B(\mu)\right)_{b_1b_2}.
\end{align*}
Here $\log(\alpha)=\sum_{k=1}^{+\infty}\frac{(-1)^{k+1}}{k}(\alpha-1)^k$ stands for its power series at $\alpha=1$.
\end{proposition}
\begin{proof}

The explicit form of the series for Hamiltonian reads
\begin{align}
H(\lambda)=\sum_{k=1}^{+\infty}\frac{(-1)^{k+1}c_k}{k+1} \overline{(B(\lambda)-\mathbf{1})^{k+1}},\qquad\textrm{where}\quad c_k=\sum_{j=1}^k{\frac1j}.
\label{eq:InterpolatingHamiltonianSeries}
\end{align}
On the other hand, from Theorem \ref{th:LaxEquation} we get
\begin{subequations}
\begin{align}
\left\{\overline{\mathrm{tr}\,(B(\lambda)-\mathbf{1})^k}, B(\mu)\right\} =&\widetilde{M}(\lambda,\mu;k)B(\mu)-B(\mu)\widetilde{M}(\lambda,\mu,k),
\end{align}
where
\begin{align}
\widetilde{M}(\lambda,\mu;k)_{b_1b_2}=k\rho_{b_2b_1}(\lambda,\mu)\left(B(\lambda) (B(\lambda)-\mathbf 1)^{k-1}\right)_{b_1b_2}.
\end{align}
\label{eq:ShiftedPowerHamiltonianFlow}
\end{subequations}
Combining (\ref{eq:InterpolatingHamiltonianSeries}) with (\ref{eq:ShiftedPowerHamiltonianFlow}) we obtain
\begin{align*}
\left\{H(\lambda),B(\mu)\right\}=\widetilde{M}(\lambda,\mu)B(\mu) -B(\mu)\widetilde{M}(\lambda,\mu),
\end{align*}
where
\begin{equation}
\begin{aligned}
\widetilde{M}(\lambda,\mu)_{b_1b_2}=&\sum_{k=1}^{+\infty}\frac{(-1)^{k+1}c_k}{k+1} \widetilde{M}(\lambda,\mu;k+1)_{b_1b_2}=\rho_{b_2b_1}(\lambda,\mu) \left(\sum_{k=1}^{+\infty}(-1)^{k+1}c_kB(\mu)(B(\mu)-\mathbf 1)^k\right)_{b_1b_2}\\
=&\rho_{b_2b_1}(\lambda,\mu)\left(\log B(\mu)\right)_{b_1b_2}.
\end{aligned}
\label{eq:InterpolatingM}
\end{equation}
\end{proof}

\begin{corollary}
When $\Gamma$ is a network on a disk (i.e. boundary measurement matrix $B(\lambda)=B$ doesn't depend on spectral parameter) we have
\begin{equation}
\begin{aligned}
\{H,B\}=MB-BM,\qquad\textrm{where}\quad
M_{ij}=\left\{\begin{array}{cc}
(\log B)_{ij},&i<j,\\[5pt]
-(\log B)_{ij},&i>j.
\end{array}\right.
\end{aligned}
\label{eq:PlanarRefactorizationFlow}
\end{equation}
\end{corollary}
\begin{proof}
Combine (\ref{eq:InterpolatingM}) with (\ref{eq:RMatrixDisc}).
\end{proof}

Now, consider a matrix
\begin{align*}
G(t)=\mathrm e^{t\log B_0}\quad\in\quad\hat{\mathcal C}[[t]].
\end{align*}
Note that we have $G(t)=\mathbf 1+O(t)$. In particular, all diagonal entries are invertible and using the standard Gaussian elimination process we can factor $G(t)$ as
\begin{align}
G(t)=\mathrm e^{t\log B_0}=\left(g_+(t)\right)^{-1}g_-(t)
\label{eq:gpgmDefinition}
\end{align}
where $g_+(t)$ is unimodular upper triangular and $g_-(t)$ is lower triangular.
\begin{lemma}
The following series
\begin{align}
B_t=g_+(t)B_0\left(g_+(t)\right)^{-1}=g_-(t)B_0\left(g_-(t)\right)^{-1}
\label{eq:BtSolutionSeries}
\end{align}
is a solution to
\begin{align*}
\frac{\mathrm d}{\mathrm dt}B_t=\{H,B_t\}.
\end{align*}
with initial condition $B=B_0$ being the boundary measurement matrix.
\end{lemma}
\begin{proof}
Repeat calculation from the proof of Theorem 2.2 in \cite{ReymanSemenov-Tian-Shansky'1994} using (\ref{eq:PlanarRefactorizationFlow}).
\end{proof}

We leave beyond the scope of the current paper a rather delicate issue of convergence of power series $g_+(t)$ (equivalently $g_-(t)$) when evaluated at a finite value of $t$. We will remark, however, that when it does converge at $t=1$, then the result must be equal to
\begin{align*}
B_1=g_+(1)B_0\left(g_+(1)\right)^{-1}=g_+(1)G(1)\left(g_+(1)\right)^{-1} =g_-(1)\left(g_+(1)\right)^{-1},
\end{align*}
which is a refactorization of the initial boundary measurement matrix
\begin{align*}
B_0=(g_+(1))^{-1}g_-(1).
\end{align*}

\begin {remark}
    A.Izosimov has pointed out, in \cite{Izosimov'2018}, that the pentagram map (and many of its generalizations) can be realized as a discretization
    of the refactorization dynamics discussed above. This was explained in the context of the weighted directed graphs considered in the present paper
    both in \cite{GSTV'2016} (for the commutative case) and in \cite{Ovenhouse'2020} (for the noncommutative case). More specifically, the space of weights
    of a particular directed graph on a cylinder is considered as the phase space, and the pentagram map itself is realized as a refactorization
    of the boundary measurement matrix. Pictorially, this means the cylinder on which the graph is drawn is cut into two, and re-glued so as to
    swap the left and right pieces.
\end {remark}

\section{Discussion and Future Directions}

An $r$-matrix formalism plays the central role in the theory of classical and quantum integrable systems.
It is natural to expect that noncommutative $r$-matrix formalism we develop in the current paper plays similar role in the theory of
noncommutative integrable systems \cite{GelfandDorfman'1981, DorfmanFokas'1992, EtingofGelfandRetakh'1997, EtingofGelfandRetakh'1998, MikhailovSokolov'2000, RetakhRubtsov'2010}.
In particular, it would be interesting to understand if noncommutative $r$-matrix formalism can be used in application to
Kontsevich system \cite{Kontsevich'2011, EfimovskayaWolf'2012, Arthamonov'2015}. At the same time, we expect that (some generalization of)
$r$-matrix formalism can be used for a more general networks with loops and/or interlaced inputs and outputs, as well as for networks on a higher genus surfaces.

Another important question which was left beyond the scope of the current paper is the study of induced Hamilton systems on the moduli space of graph connections.
Double (quasi) Poisson brackets are in line with the so-called Kontsevich-Rosenberg principle \cite{Kontsevich'1993} of Noncommutative Geometry.
Namely, double (quasi) Poisson bracket on associative algebra $A$ induces usual (quasi) Poisson brackets on representation varieties
$\mathrm{Hom}(A,\mathrm{Mat}_n(\mathbb C))$ for all $n\in\mathbb N$. Elements of the cyclic space in involution give rise to commuting Hamilton
flows on representation variety. In particular, the systems considered in our paper give rise to a family of commuting Hamilton flows on the moduli
space of $GL(n,\mathbb C)$ graph connections for all $n\in\mathbb N$. It would be interesting to determine whether the induced systems are integrable in the Liouville sense,
i.e. have enough independent hamiltonians in involution.

From a purely algebraic point of view, it would be interesting to study $r$-matrix double brackets (\ref{eq:TwistedCommutatorFormal}) for different $r$-matrices satisfying (\ref{eq:QuasiYangBaxter}).

Last but not least, in a preprint \cite{Izosimov'2021}, which appeared after the initial version of this manuscript was already prepared, it was shown that usual (commutative) dimer cluster integrable systems \cite{GoncharovKenyon'2011} are equivalent to the cluster integrable systems associated perfect networks \cite{GSTV'2016}. It is an interesting open problem to generalize results of \cite{Izosimov'2021} to the noncommutative case and examine relationship between noncommutative cluster integrable systems studied in this paper and construction of \cite{GoncharovKontsevich'2021}.

\bibliographystyle{alpha}
\bibliography{references}

\end{document}